\documentclass[12pt]{amsart}

\usepackage{comment}
\usepackage[initials]{amsrefs}

\pdfoutput=1
\usepackage[margin=2.5cm]{geometry}
\usepackage[all]{xy}
\usepackage{longtable}
\usepackage{amssymb}

\usepackage{mathrsfs}  
\usepackage{enumerate}
\usepackage{hyperref}
\usepackage{tikz-cd}

\usepackage{tikz}
\tikzset{thickstyle/.style={shape=circle,fill=black,scale=0.3}}
\tikzset{thinstyle/.style={shape=circle,fill=black,scale=0.15}}

\theoremstyle{plain}
\newtheorem{thm}{Theorem}[section] 
\newtheorem{pro}[thm]{Proposition}
\newtheorem{lem}[thm]{Lemma}

\newtheorem*{conA}{Conjecture A}
\newtheorem*{conB}{Conjecture B}

\theoremstyle{definition}
\newtheorem{dfn}[thm]{Definition}
\newtheorem{rem}[thm]{Remark}
\newtheorem{exa}[thm]{Example}

\newcommand{\lan}{\langle}
\newcommand{\ran}{\rangle}

\newcommand{\hull}[1]{\mathrm{conv}\left( #1 \right)}

\newcommand{\vectortwo}[2]{{\begin{pmatrix} #1 \\ #2 \end{pmatrix}}}
\newcommand{\vectorthree}[3]{{\begin{pmatrix} #1 \\ #2 \\ #3 \end{pmatrix}}}
\newcommand{\vectorfour}[4]{{\begin{pmatrix} #1 \\ #2 \\ #3 \\ #4\end{pmatrix}}}

\newcommand{\seeds}{\mathscr{S}}

\newcommand{\usp}{u} 

\newcommand{\divV}{\partial V} 

\newcommand{\moduli}[1]{\mathscr{M}_{#1}}

\DeclareMathOperator{\coker}{coker}

\DeclareMathOperator{\divi}{div}
\DeclareMathOperator{\Defo}{Def}

\DeclareMathOperator{\Hom}{Hom}
\DeclareMathOperator{\Ker}{Ker}
\DeclareMathOperator{\Newt}{Newt}
\DeclareMathOperator{\mut}{mut}
\DeclareMathOperator{\Ext}{Ext}
\DeclareMathOperator{\NE}{NE}

\DeclareMathOperator{\Aff}{Aff} 
\DeclareMathOperator{\Spec}{Spec}

\DeclareMathOperator{\Gm}{\mathbb{G}_m}

\newcommand{\degs}{J}
\newcommand{\into}{\hookrightarrow}

\newcommand{\Wb}{W^\flat}

\newcommand{\Db}{D^\flat}

\newcommand{\Bbar}{\overline{B}}
\renewcommand{\Dbar}{\overline{D}}
\newcommand{\Wbar}{\overline{W}}
\newcommand{\Xbar}{\overline{X}}

\renewcommand{\AA}{\mathbb{A}}

\newcommand{\RR}{\mathbb{R}}
\newcommand{\PP}{\mathbb{P}}
\newcommand{\ZZ}{\mathbb{Z}}
\newcommand{\NN}{\mathbb{N}}

\newcommand{\CC}{\mathbb{C}}

\newcommand{\LL}{\mathbb{L}}
\newcommand{\TT}{\mathbb{T}}
\newcommand{\cD}{\mathscr{D}}
\newcommand{\cL}{\mathcal{L}}
\newcommand{\cO}{\mathcal{O}}
\newcommand{\cU}{\mathscr{U}}
\newcommand{\cM}{\mathscr{M}}

\newcommand{\rmi}{\mathrm{i}}
\newcommand{\rmd}{\mathrm{d}}

\title[Mirror Symmetry and Smoothing Gorenstein toric affine \mbox{$3$-folds}]{Mirror Symmetry and Smoothing Gorenstein toric affine \mbox{$3$-folds}}

\author{Alessio Corti}
\address{Department of Mathematics,
 Imperial College London,
180 Queen's Gate, SW7~2AZ
 London, United Kingdom}
\email{a.corti@imperial.ac.uk}

\author{Matej Filip}
\address{Department of Mathematics,
 University of Ljubljana,
 Trzaska cesta 25, Ljubljana, Slovenia}
\email{matej.filip@fe.uni-lj.si}

\author{Andrea Petracci}
\address{Institut f\"ur Mathematik,
Freie Universit\"at Berlin, Arnimallee 3, 14195 Berlin, Germany}
\email{andrea.petracci@fu-berlin.de}

\begin{document}

\begin{abstract}
  We state two conjectures that together allow one to describe the set
  of smoothing components of a Gorenstein toric affine \mbox{$3$-fold}
  in terms of a combinatorially defined and easily studied set of
  Laurent polynomials called \emph{$0$-mutable polynomials}.  We
  explain the origin of the conjectures in mirror symmetry and present
  some of the evidence.
\end{abstract}

\maketitle

\tableofcontents

\section{Introduction}
\label{sec:introduction}

We explore mirror symmetry for smoothings of a $3$-dimensional
Gorenstein toric affine variety $V$. Specifically, we try to imagine
what consequences mirror symmetry may have for the classification of
smoothing components of the deformation space $\Defo V$. Conjecture~A
makes the surprising statement that the set of smoothing components of
$\Defo V$ is in bijective correspondence with a set of easily defined
and enumerated 2-variable Laurent polynomials, called $0$-mutable
polynomials. Our Conjecture~B --- in the strong form stated in
Remark~\ref{rem:conjBStrongForm} --- asserts that these smoothing
components are themselves smooth, and computes their tangent spaces
from the corresponding $0$-mutable polynomials.

\medskip

As is customary in toric geometry, $V$ is associated to a strictly
convex 3-dimensional rational polyhedral cone
$\sigma \subseteq N_\RR$, where $N$ is a $3$-dimensional lattice; the
Gorenstein condition means that the integral generators of the rays of
$\sigma$ all lie on an integral affine hyperplane $(\usp = 1)$ for
some $\usp \in M := \Hom_\ZZ(N, \ZZ)$.  We denote by $F$ the convex
hull of the integral generators of the rays of $\sigma$, i.e.\
\[
F := \sigma \cap (\usp=1);
\]
this is a lattice polygon (i.e.\ a lattice polytope of dimension $2$)
embedded in the affine $2$-dimensional lattice $(u=1)$.  The
isomorphism class of the toric variety $V$ depends only on the
affine equivalence class of the polygon $F$.

If $V$ has an isolated singularity, then $\Defo V$ is finite dimensional and
we know from the work of
Altmann~\cite{altmann_versal_deformation} that there is a $1$-to-$1$
correspondence between the set of irreducible components of $\Defo V$
and integral maximal Minkowski decompositions of $F$.
Altmann also shows that, when taken with their
\emph{reduced} structure, these components are all themselves smooth. 

We are interested in the case when $V$ has non-isolated
singularities. Very little is known at this level of generality, but
examples show that the picture for non-isolated singularities is very
different from the one just sketched for isolated singularities.  Our
main reason for wanting to work with non-isolated singularities is the
Fanosearch project: we wish to prove a general criterion for smoothing
a toric Fano variety, and Conjecture~A here is just the local case.

Conjecture~A characterizes \emph{smoothing components} of $\Defo V$ in
terms of the combinatorics of the polygon $F$. Specifically, we define
the set $\mathfrak{B}$ of $0$-mutable Laurent polynomials with Newton polygon $F$,
and the conjecture states that there is a canonical bijective
correspondence $\kappa \colon \mathfrak{B}\to \mathfrak{A}$, where
$\mathfrak{A}$ is the set of smoothing components of $\Defo V$.

At first sight the formulation of the conjecture seems strange;
however, the statement makes sense in the context of mirror symmetry,
where (conjecturally) the $0$-mutable polynomials are the mirrors of
the corresponding smoothing components.

In Section~\ref{sec:intr-mirr-symm}, we state a new
Conjecture~B,\footnote{The statement of Conjecture~B comes after
	Sec.~\ref{sec:properties-0-mutable} but does not logically depend on
	it: if you wish, you can skip directly from
	Sec.~\ref{sec:def-of-0-mutable} to Sec.~\ref{sec:conjectures}.}
which implies the existence of a map
$\kappa \colon \mathfrak{B}\to \mathfrak{A}$ --- see
Remark~\ref{rem:conjBimpliesmapdefined}. In that section, we also explain how to
(conjecturally) construct a deformation directly from a $0$-mutable
polynomial in the spirit of the intrinsic mirror symmetry of
Gross--Siebert~\cite{1909.07649, gs_intrinsic_slc} and work of
Gross--Hacking--Keel~\cite{MR3415066}.

The coefficients of the
$0$-mutable Laurent polynomials that appear in our conjecture ought
themselves to enumerate certain holomorphic discs in the corresponding
smoothing, and we would love to see a precise statement along these
lines.

In our view, the conjectures together are nothing other than a
statement of mirror symmetry as a one-to-one correspondence between
two sets of objects, similar to the conjectures made in
\cite{MR3430830} in the context of orbifold del Pezzo surfaces, and
the correspondence between Fano 3-folds and Minkowski polynomials
discovered in \cite{MR3470714}.

In Section~\ref{sec:properties-0-mutable} we give some equivalent
characterisations of $0$-mutable polynomials and begin to sketch some
of their general properties. These properties make it very easy to
enumerate the $0$-mutable polynomials with given Newton polygon.
The material here is rather sketchy --- full details will appear elsewhere;
it serves for context, but it is not logically necessary for the statement
of the conjectures.

The suggestion that there is a simple structure to the set of
smoothing components is surprising in a subject that --- as all serious
practitioners know --- is marred by Murphy's law. In fact, there is a
substantial body of direct and circumstantial evidence for the
conjectures, some of which we present in
Section~\ref{sec:evidence-1}.

In the final Section~\ref{sec:evidence-matej} we compute in detail the
deformation space of the variety $V_F$ associated to the polygon $F$ of
Example~\ref{ex:quadrilateral}, giving evidence for the
conjectures. Some of the reasons for choosing this particular example
are:
\begin{enumerate}
	\item The variety $V_F$ is of codimension $5$ and hence it lies outside the
	--- still rudimentary but very useful --- structure theory of codimension-$4$
	Gorenstein rings \cite{MR3380790}, see also \cite{1812.02594, tom_and_jerry};
	\item For this reason, $V_F$ is a good test of the technology of
	\cite{fil, fil2} as a tool for possibly proving the conjectures;
	\item The polygon $F$ appears as a facet of some of the
	$3$-dimensional reflexive polytopes and hence it is immediately
	relevant for the Fanosearch project.
\end{enumerate}
As things stand, we are some distance away from being able to prove
the conjectures. We had a tough time even with the example of
Section~\ref{sec:evidence-matej}: while a treatment based on
\cite{fil, fil2} seems possible, the task became so tedious that we
decided instead to rely on Ilten's Macaulay2 \cite{M2} package
\emph{Versal deformations and local Hilbert schemes}
\cite{ilten_package_article}. That package makes it possible to test
the conjectures in many other examples in codimension $\geq 5$.

\subsection*{Notation and conventions}
We work over $\CC$, but everything holds over an algebraically closed field of characteristic zero.
We refer the reader to \cite{fulton_book} for an introduction to toric geometry.
All the toric varieties we consider are normal.
We use the following notation.

\begin{longtable}{lp{0.85\textwidth}}
	$F$ & a lattice polygon \\
	$V$ & the Gorenstein toric affine \mbox{$3$-fold} associated to the cone over $F$ put at height 1\\
	$\divV$ & the toric boundary of $V$ \\
	$\Xbar$ & the projective toric surface associated to the normal fan of $F$\\
	$\Bbar$ & the toric boundary of $\Xbar$ \\
	$A$ & the ample line bundle on $\Xbar$ given by $F$ \\
	$X$ & the cluster surface associated to $F$ (the non-toric blowup of
	$\Xbar$ constructed in Section~\ref{sec:properties-0-mutable})\\
	$B$ & the strict transform of $\Bbar$ in $X$ \\
	$\Wb$ & the toric \mbox{$3$-fold} constructed in Section~\ref{sec:intr-mirr-symm}\\
	$\Wbar$ & the toric blowup of $\Wb$ constructed in Section~\ref{sec:intr-mirr-symm}\\
	$W$ & the mirror cluster variety (the non-toric blowup of $\Wbar$
	constructed in Section~\ref{sec:intr-mirr-symm})
\end{longtable}

\subsection*{Thanks to Bill from AC} I was lucky to be a L.E.\ Dickson
Instructor at the University of Chicago in the years 1993--1996, where
I worked in the group led by Bill Fulton. I had studied deformations
of singularities in a seminar held in 1988--89 at the University of
Utah, but in Bill's seminar I learnt about Chow groups and quantum
cohomology. It was a wonderful time in my work life thanks largely to
Bill. This paper features many of the ideas that I learnt in Bill's
seminar and it is very nice to see that they are so relevant in
the study of deformations of singularities.

\subsection*{Acknowledgements} We thank Klaus Altmann, Tom Coates,
Mark Gross, Paul Hacking, Al Kasprzyk, Giuseppe Pitton, and Thomas
Prince for many helpful conversations.

We particularly thank Al Kasprzyk for sharing with us and allowing us
to present some of his unpublished ideas on maximally mutable Laurent
polynomials \cite{KT}.

We discussed with Mark Gross some of our early experiments with
$0$-mutable polynomials.

We owe special thanks to Paul Hacking who read and corrected various
mistakes in earlier versions of the paper --- the responsibility for
the mistakes that are left is of course ours.

Giuseppe Pitton ran computer calculations that provide indirect
evidence for the conjectures.

The idea that mirror Laurent polynomials are characterised by their
mutability goes back to Sergey Galkin \cite{galkin_usnich}.

It will be clear to all those familiar with the issues that this paper
owes a very significant intellectual debt to the work of Gross,
Hacking and Keel \cite{MR3415066} and the intrinsic mirror symmetry of
Gross and Siebert \cite{1909.07649, gs_intrinsic_slc}.

Last but not least, it is a pleasure to thank the anonymous referee
who read our manuscript very carefully, and told us about Ilten's
Macaulay2 package.

\section{Conjecture~A}
\label{sec:main}

\subsection{Gorenstein toric affine varieties}

Consider a rank-$n$ lattice $N\simeq \ZZ^n$ (usually $n=3$) and, as
usual in toric geometry, its dual lattice $M=\Hom_\ZZ (N,\ZZ)$.  The
$n$-dimensional torus
\[
\TT= \Spec \CC [M]
\]
is referred to simply as ``the'' torus.\footnote{Sometimes we denote
	this torus by $\TT_N$, that is, the commutative group scheme
	$N\otimes_\ZZ \Gm$ such that for all rings $R$
	$\TT_N(R)=N\otimes_{\ZZ} R^\times$.} Consider a strictly convex
full-dimensional rational polyhedral cone $\sigma \subset N_\RR$ and
the corresponding affine toric variety
\[
V = \Spec \CC[ \sigma^\vee \cap M].
\]
This is a normal Cohen--Macaulay $n$-dimensional variety.

By definition $V$ is Gorenstein if and only if the pre-dualising sheaf 
\[
\omega^0_V = \mathcal{H}^{-n} (\omega_V^\bullet)
\]
is a line bundle. Since our $V$ is Cohen--Macaulay, this is the same as insisting
that all the local rings of $V$ are local Gorenstein rings.
It is known and not difficult to show that $V$ is Gorenstein if and only if there is a
vector $\usp \in M$ such that the integral generators
$\rho_1, \dots, \rho_m$ of the rays of the cone $\sigma$ all lie on the affine
lattice $\LL=(\usp = 1)\subset N$. Such vector $\usp$ is called the \emph{Gorenstein degree}.

If $V$ is Gorenstein, then the toric boundary of $V$ is the following effective reduced
Cartier divisor on $V$:
\[
\divV = \Spec \CC[\sigma^\vee \cap M]/(x^\usp).
\]

If $V$ is Gorenstein, we set
\[
F := \sigma \cap (\usp=1);
\]
this is an $(n-1)$-dimensional lattice polytope embedded in the affine
lattice $\LL = (u=1)$ and it is the convex hull of the integral
generators of the rays of the cone $\sigma$.  One can prove that the
isomorphism class of $V$ depends only on the affine equivalence class
of the lattice polytope $F$ in the affine lattice $\LL$.  Therefore we
will say that $V$ is associated to the polytope $F$.

This establishes a $1$-to-$1$ correspondence between isomorphism
classes of Gorenstein toric affine \mbox{$n$-folds} without torus
factors and $(n-1)$-dimensional lattice polytopes up to affine
equivalence.

\subsection{Statement of Conjecture~A}
\label{sec:statement-A}

In this section we explain everything that is needed to make sense of
the following:

\begin{conA}
	Consider a lattice polygon $F$ in a $2$-dimensional affine lattice $\LL$.
	Let $V$ be the Gorenstein toric affine \mbox{$3$-fold} associated to $F$.
	
	Then there is a canonical bijective function $\kappa \colon \mathfrak{B}\to \mathfrak{A}$ where:
	\begin{itemize}
		\item $\mathfrak{A}$ is the set of smoothing components of the miniversal deformation space
		$\Defo V$,
		\item $\mathfrak{B}$ is the set of $0$-mutable polynomials $f\in \CC[\LL]$ with Newton
		polygon $F$.
	\end{itemize}
\end{conA}

\begin{rem}
	\label{rem:infinite_dimensional}
	It is absolutely crucial to appreciate that we are not assuming that
	$V$ has an isolated singularity at the toric $0$-stratum. If $V$
	does not have isolated singularities, $\Defo V$ is
	infinite-dimensional. A few words are in order to clarify what kind
	of infinite dimensional space $\Defo V$ is.
	
	In full generality, there is some discussion of this issue in the
	literature on the analytic category, see for example~\cite{MR713093,
		MR803194}.\footnote{We thank Jan Stevens for pointing out these
		references to us. We are not aware of a similar discussion in the
		algebraic literature.} In the special situation of interest in
	this paper, we take a na\"{\i}ve approach, which we briefly explain,
	based on the following two key facts:
	\begin{enumerate}[(i)]
		\item If $V$ is a Gorenstein toric affine 3-fold, then $T^2_V$ is
		finite dimensional. Indeed $V$ has transverse
		$A_\star$-singularities in codimension two, hence it is
		unobstructed in codimension two, hence $T^2_V$ is a finite length
		module supported on the toric $0$-stratum. In fact, there is an
		explicit description of $T^2_V$ as a representation of the torus,
		see \cite[Section~5]{altmann_andre_quillen_cohomology}, an example
		of which is in Lemma~\ref{lem t2} below. This shows that $\Defo V$
		is cut out by finitely many equations.
		\item On the other hand, the known explicit description of $T^1_V$
		as a representation of the torus
		\cite[Theorem~4.4]{altmann_one_parameter}, together with the
		explicit description of $T^2_V$ just mentioned, easily implies
		that each of the equations can only use finitely many
		variables.\footnote{The key point is to show that for all fixed
			weights $\mathbf{m}\in M$, the equation
			$\mathbf{m}=\sum m_i \mathbf{v}_i$ has finitely many solutions
			for $m_i\in \NN\setminus \{0\}$ and $\mathbf{v}_i\in M$ a weight
			that appears non-trivially in $T^1_V$. This type of
			consideration is used extensively in the detailed example
			discussed in Section~\ref{sec:evidence-matej}.}
	\end{enumerate}
	Thus we can take $\Defo V$ to be the $\Spec$ of a non-Noetherian
	ring, that is, the simplest kind of infinite-dimensional
	scheme.\footnote{The situation is not so simple for the universal
		family $\cU\to \Defo V$. Indeed the equations of $\cU$ naturally
		involve all the infinitely many coordinate functions on
		$T^1_V$. Thus, $\cU$ is a bona fide ind-scheme. The language to
		deal with this exists but it is not our concern here.}
\end{rem}

\begin{rem}
	\label{rem:1}
	Conjecture~A does not state what the function $\kappa$ is, nor what
	makes it ``canonical.'' The existence of a function $\kappa \colon
	\mathfrak{B}\to \mathfrak{A}$ is implied by Conjecture~B stated in
	Section~\ref{sec:conjectures} below, see Remark~\ref{rem:conjBimpliesmapdefined}.
\end{rem}

\begin{rem} \label{rem:deformations_variety_or_pair}
	Let $V$ be a Gorenstein toric affine \mbox{$3$-fold} and let $\divV$
	be the toric boundary of $V$.  Let $\Defo (V, \divV)$ be the
	deformation functor (or the base space of the miniversal
	deformation) of the pair $(V, \divV)$.  There is an obvious
	forgetful map $\Defo (V, \divV) \to \Defo V$. In this case, since
	$\divV$ is an effective Cartier divisor in $V$ and $V$ is affine,
	this map is smooth of relative dimension equal to the dimension of
	$\coker \left( H^0(\theta_V) \to H^0(N_{\divV / V}) \right)$, where $\theta_V$ is
	the sheaf of derivations on $V$ and $N_{\divV / V} = \cO_{\partial V}(\partial V)$ is the normal
	bundle of $\divV$ inside $V$.  In particular, this implies that
	$\Defo V$ and $\Defo (V, \divV)$ have exactly the same irreducible
	components.  One can also see that a smoothing component in
	$\Defo V$ is the image of a component of $\Defo (V, \divV)$ where
	also $\divV$ is smoothed. 
	
	In other words, we can equivalently work with deformations of $V$
	or deformations of the pair $(V,\divV)$. The right thing to do in
	mirror symmetry is to work with deformations of the pair $(V,\divV)$; however, 
	the literature on deformations of singularities is all written in
	terms of $V$. In most cases it is not difficult to make the translation but this
	paper is not the right place for doing that. Thus when possible we
	work with deformations of $V$. The
	formulation of Conjecture~B in Section~\ref{sec:intr-mirr-symm}
	requires that we work with deformations of the pair $(V, \divV)$. 
\end{rem}

\subsection{The definition of $0$-mutable Laurent polynomials}
\label{sec:def-of-0-mutable}

This subsection is occupied by the definition of $0$-mutable Laurent
polynomials. The simple key idea --- and the explanation for the name
``$0$-mutable'' --- is that an irreducible Laurent polynomial $f$ is
$0$-mutable if and only if there is a sequence of mutations
\[
f\mapsto f_1\mapsto \cdots \mapsto f_p=\mathbf{1}
\]
starting from $f$ and ending with the constant monomial
$\mathbf{1}$. The precise definition is given below after some
preliminaries on mutations. The definition is appealing and it is
meaningful in all dimensions, but it is not immediately useful if you
want to study $0$-mutable polynomials. Indeed, for example, to prove
that a given polynomial $f$ is $0$-mutable one must produce a chain of
mutations as above and it may not be obvious where to look. It is even
less clear how to prove that $f$ is not $0$-mutable. In
Section~\ref{sec:properties-0-mutable},
Theorem~\ref{thm:equivalent_defns}, we prove two useful
characterizations of $0$-mutable polynomials in two variables. The
first of these states that a polynomial is $0$-mutable if and only if
it is rigid maximally mutable. From this property it is easy to check
that a given polynomial is $0$-mutable, that it is not $0$-mutable,
and to enumerate $0$-mutable polynomials with given Newton
polytope. The second characterization states that a (normalized, see
below) Laurent polynomial in two variables is $0$-mutable if and only
if the irreducible components of its vanishing locus are $-2$-curves
on the cluster surface: see Section~\ref{sec:properties-0-mutable} for
explanations and details.

\medskip

Let $\LL$ be an affine lattice and let $\LL_0$ be its underlying
lattice.\footnote{In our setup $\LL=(\usp =1)\subset N$ does not have
	a canonical origin. We try to be pedantic and write $\LL_0$ for the
	the underlying lattice -- in our example, $\LL_0=\Ker \usp$. In
	practice this distinction is not super-important and you are free to
	choose an origin anywhere you want.}

In other words, $\LL_0$ is a free abelian group of finite rank and
$\LL$ is a set together with a free and transitive $\LL_0$-action.  We
denote by $\CC[\LL]$ the vector space over $\CC$ whose basis is made
up of the elements of $\LL$.  For every $l \in \LL$ we denote by $x^l$
the corresponding element in $\CC[\LL]$.  Elements of $\CC[\LL]$ will
be called (Laurent) polynomials.  It is clear that $\CC[\LL]$ is a
rank-$1$ free module over the $\CC$-algebra $\CC[\LL_0]$.  The choice
of an origin in $\LL$ specifies an isomorphism
$\CC[\LL] \simeq \CC[\LL_0]$.

\begin{dfn}
	\label{dfn:normalised} A Laurent polynomial $f\in \CC[\LL]$ is
	\emph{normalized} if for all vertices $v\in \Newt f$ we have $a_v=1$
	where $a_v$ is the coefficient of the monomial $x^v$ as it appears
	in $f$.
\end{dfn}

In this paper all polynomials are assumed to be normalized unless
explicitly stated otherwise. 

If $f\in \CC[\LL]$ then we say that $f$
\emph{lives} on the smallest saturated affine sub-lattice $\LL^\prime\subseteq \LL$ such
that $f\in \CC[\LL^\prime]$. The property of being $0$-mutable only
depends on the lattice where $f$ lives --- here it is crucial that we only allow saturated sub-lattices.
We say that $f$ is an
$r$-variable polynomial if $f$ lives on a rank-$r$ affine lattice. 

Let us start by defining $0$-mutable polynomials in $1$ variable.

\begin{dfn} \label{dfn:0-mutable_1_variable} Let $\LL$ be an affine
	lattice of rank $1$ and let $\LL_0$ be its underlying lattice. Let
	$v \in \LL_0$ be one of the two generators of $\LL_0$. A
	polynomial $f \in \CC[\LL]$ is called \emph{$0$-mutable} if
	\[
	f = (1+x^v)^k x^l
	\]
	for some $l \in \LL$ and $k \in \NN$.
\end{dfn}

(It is clear that the definition does not depend on the choice of $v$.)

If $f$ is a Laurent polynomial in 1 variable and its Newton polytope
is a segment of lattice length $k$, then $f$ is $0$-mutable if and only if
the coefficients of $f$ are the $k+1$ binomial coefficients of weight $k$.

\medskip

The definition of $0$-mutable polynomials in more than 1 variable will
be given recursively on the number of variables.  Thus from now on we
fix $r \geq 2$ and we assume to know already what it means for a
polynomial of $<r$ variables to be $0$-mutable.  Before we can state
what it means for a polynomial $f$ of $r$ variables to be $0$-mutable, 
we need to explain how to mutate $f$.

If $\LL$ is an affine lattice, we denote by $\Aff (\LL, \ZZ)$ the
lattice of affine-linear functions $\varphi \colon \LL \to \ZZ$. 
If $\LL_0$ denotes the underlying lattice of $\LL$, $\varphi$ has a
well-defined \emph{linear part} which we denote by $\varphi_0\colon \LL_0\to \ZZ$.

\begin{dfn} 
	\label{dfn:mutation}
	Let $r\geq 2$ and fix a rank-$r$ affine lattice $\LL$.
	
	A \emph{mutation datum} is a pair $(\varphi, h)$ of a non-constant affine-linear
	function $\varphi \colon \LL \to \ZZ$ and a $0$-mutable polynomial
	$h\in \CC[\Ker \varphi_0]$.
	
	Given a mutation datum $(\varphi,h)$ and $f\in \CC[\LL]$, write (uniquely)
	\[
	f=\sum_{k\in \ZZ} f_k 
	\qquad
	\text{where}
	\qquad
	f_k \in \CC[(\varphi =k) \cap \LL]
	\]
	We say that $f$ is \emph{$(\varphi,h)$-mutable} if for all $k<0$
	$h^{-k}$ divides $f_k$ (equivalently, if for all $k \in \ZZ$
	$h^kf_k\in \CC[\LL]$). If $f$ is $(\varphi,h)$-mutable, then the
	\emph{mutation} of $f$, with respect to the mutation datum
	$(\varphi,h)$, is the polynomial:
	\[
	\mut_{(\varphi, h)} f =\sum_{k\in \ZZ} h^k f_k.
	\]
\end{dfn}

\begin{rem}
	The notion of mutation goes back (at least) to
	Fomin--Zelevinsky~\cite{fomin_zelevinsky}. We first learned of
	mutations from the work of Galkin--Usnich~\cite{galkin_usnich}
	and~Akhtar--Coates--Galkin--Kasprzyk~\cite{sigma}. This paper owes a
	significant intellectual debt to the interpretation of mutations
	developed in work by Gross, Hacking and Keel, for instance
	\cite{MR3314827, MR3350154, MR3415066}.\footnote{Many mathematicians
		work on mutations from different perspectives and we apologize for
		not even trying to quote all the relevant references here.}
\end{rem}

The following is a recursive definition. The base step is given by
Definition~\ref{dfn:0-mutable_1_variable}.

\begin{dfn}
	\label{dfn:0-mutable}
	Let $\LL$ be an affine lattice of rank $r \geq 2$. We define the set of $0$-mutable polynomials on $\LL$
	in the following recursive way.
	\begin{enumerate}[(i)]
		\item If $\LL^\prime$ is a saturated affine sub-lattice of $\LL$
		and $f \in \CC[\LL^\prime]$ is $0$-mutable, then $f$ is
		$0$-mutable in $\CC[\LL]$.
		\item If $f=f_1 f_2$ is reducible, then $f$ is 0-mutable if both
		factors $f_1$, $f_2$ are $0$-mutable.\footnote{In order to take the ``product'' $f_1f_2$
			one has to choose an origin in $\LL$. This choice makes no
			difference to the definition.}
		\item If $f$ is irreducible, then $f$ is $0$-mutable if a mutation of $f$ is $0$-mutable.
	\end{enumerate}
\end{dfn}

Equivalently, the set of $0$-mutable polynomials of $\leq r$
variables is the smallest subset of $\CC[\LL]$ that contains all $0$-mutable
polynomials of $<r$ variables and that is closed under the operations
of taking products (in particular translation) and mutations of
irreducible polynomials.

\begin{rem}
	\label{rem:2}
	It follows easily from the definition that $0$-mutable polynomials
	are normalized. Indeed:
	\begin{enumerate}[(1)]
		\item The polynomials in Definition~\ref{dfn:0-mutable_1_variable}
		(the base case) are $0$-mutable;
		\item The product of two normalized polynomials is normalized;
		\item The mutation of a normalized polynomial is normalized. 
	\end{enumerate}
\end{rem}

\begin{exa}
	
	Let $\LL$ be an affine lattice and let $v$ be a primitive vector in
	the underlying lattice $\LL_0$.
	Definition~\ref{dfn:0-mutable_1_variable} and
	Definition~\ref{dfn:0-mutable}(i) imply that $(1+x^v)^k x^l$ is
	$0$-mutable for all $l \in \LL$ and $k \in \NN$.
\end{exa}

\begin{exa} \label{exa:conePP123}
	Consider the triangle
	\begin{equation} \label{eq:triangle_123}
	F = \hull{(0,0),(3,0),(3,2)}
	\end{equation}
	in the lattice $\LL = \ZZ^2$. Let us identify $\CC[\LL]$ with
	$\CC[x^\pm, y^\pm]$.  One can prove that there are exactly two
	$0$-mutable polynomials with Newton polytope $F$, namely:
	\begin{align*}
		(1+x)^3 + 2(1+x)x^2y + x^3y
		\qquad \text{and} \qquad
		(1+y)^2 x^3 + 3 (1+y)x^2 + 3x + 1.
	\end{align*}
	In Figure~\ref{fig:123} we have written the coefficients of these two polynomials
	next to the lattice point of $F$ associated to the corresponding
	monomial.
	
	It is shown in \cite{1812.02594} that $\Defo V$ has two components, and
	that they are both smoothing components, confirming our
	conjectures. The calculation there goes back to unpublished work by Jan
	Stevens, but see also~\cite{altmann_one_parameter}.
\end{exa}

\begin{figure}
	\centering
	\begin{tikzpicture}
	\foreach \x in {0, ..., 3}
	\foreach \y in {0, ..., 0}
	{
		\node[thickstyle] (\x - \y) at (\x, \y){};
	}
	\node[thickstyle] (3 - 1) at (3, 1){};
	\node[thickstyle] (3 - 2) at (3, 2){};
	\node[thickstyle] (2 - 1) at (2, 1){};
	\node[thinstyle] (0 - 1) at (0, 1){};
	\node[thinstyle] (1 - 1) at (1, 1){};
	\node[thinstyle] (0 - 2) at (0, 2){};
	\node[thinstyle] (1 - 2) at (1, 2){};
	\node[thinstyle] (2 - 2) at (2, 2){};
	\draw[thick] (0, 0) -- (3, 0) -- (3, 2) -- cycle;
	\node [anchor=east] at (0.25, 0.25) {$1$};
	\node [anchor=east] at (1.5, 0.25) {$3$};
	\node [anchor=east] at (2.5, 0.25) {$3$};
	\node [anchor=east] at (3.5, 0.25) {$1$};
	\node [anchor=east] at (2.5, 1.25) {$2$};
	\node [anchor=east] at (3.5, 1.25) {$2$};
	\node [anchor=east] at (3.5, 2.25) {$1$};
	\end{tikzpicture}
	\qquad \qquad
	\begin{tikzpicture}
	\foreach \x in {0, ..., 3}
	\foreach \y in {0, ..., 0}
	{
		\node[thickstyle] (\x - \y) at (\x, \y){};
	}
	\node[thickstyle] (3 - 1) at (3, 1){};
	\node[thickstyle] (3 - 2) at (3, 2){};
	\node[thickstyle] (2 - 1) at (2, 1){};
	\node[thinstyle] (0 - 1) at (0, 1){};
	\node[thinstyle] (1 - 1) at (1, 1){};
	\node[thinstyle] (0 - 2) at (0, 2){};
	\node[thinstyle] (1 - 2) at (1, 2){};
	\node[thinstyle] (2 - 2) at (2, 2){};
	\draw[thick] (0, 0) -- (3, 0) -- (3, 2) -- cycle;
	\node [anchor=east] at (0.25, 0.25) {$1$};
	\node [anchor=east] at (1.5, 0.25) {$3$};
	\node [anchor=east] at (2.5, 0.25) {$3$};
	\node [anchor=east] at (3.5, 0.25) {$1$};
	\node [anchor=east] at (2.5, 1.25) {$3$};
	\node [anchor=east] at (3.5, 1.25) {$2$};
	\node [anchor=east] at (3.5, 2.25) {$1$};
	\end{tikzpicture}
	\caption{The two $0$-mutable polynomials whose Newton polytope is the triangle $F$ defined in \eqref{eq:triangle_123}}
	\label{fig:123}
\end{figure}
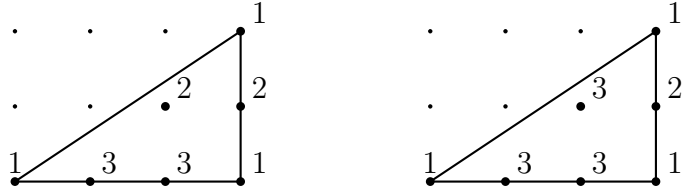

\begin{exa}\label{ex:quadrilateral}
	Consider the quadrilateral
	\begin{equation} \label{eq:quadrilateral}
	F = \hull{(-1,-1),(2,-1),(1,1),(-1,2)}
	\end{equation}
	in the lattice $\LL = \ZZ^2$. Let us identify $\CC[\LL]$ with $\CC[x^\pm, y^\pm]$.
	Consider the polynomial
	\[
	g = \frac{(1+x)^3 + (1+y)^3 - 1 + x^2 y^2 }{xy},
	\]
	which is obtained by giving binomial coefficients to the lattice
	points of the boundary of $F$ and by giving zero coefficient to the
	interior lattice points of $F$.  By Lemma~\ref{lem:easy_properties}(2)
	every $0$-mutable polynomial with Newton polytope $F$ must coincide
	with $g$ on the boundary lattice points of $F$.  One can prove that
	there are exactly three $0$-mutable polynomials with Newton polytope
	$F$, namely:
	\begin{gather*}
		\alpha = g + 5 + 2x + 2y = \frac{(1+x+2y+y^2)(1+2x+x^2+y)}{xy}, \\
		\beta = g + 6 + 3x + 4y = \frac{(1+x)^3 + 3y(1+x)^2 + y^2(1+x)(3+x) + y^3}{xy}, \\
		\gamma = g + 6 + 4x + 3y = \frac{(1+y)^3 + 3 x (1+y)^2 + x^2 (1+y)(3+y) + x^3}{xy}.
	\end{gather*}
	In Figure~\ref{fig:114} we have written down the coefficients of these
	three polynomials.
	The polynomial $\alpha$ is reducible and it is easy to show that its factors are $0$-mutable.
	
	\begin{figure}
		\centering
		\begin{tikzpicture}
		\foreach \x in {0,1,2}
		\foreach \y in {0,1,2}
		{
			\node[thickstyle](\x - \y) at (\x, \y){};
		}
		\node[thickstyle](3 - 0) at (3, 0){};
		\node[thickstyle](0 - 3) at (0, 3){};
		\draw[thick] (0, 0) -- (3, 0) -- (2, 2) -- (0, 3) -- cycle;
		\node [anchor=north] at (0, 0) {$1$};
		\node [anchor=north] at (1,0) {$3$};
		\node [anchor=north] at (2,0) {$3$};
		\node [anchor=north] at (3,0) {$1$};
		\node[anchor=east] at (0,1) {$3$};
		\node[anchor=east] at (0,2) {$3$};
		\node[anchor=east] at (0,3) {$1$};
		\node[anchor=south] at (2.2,1.9) {$1$};
		\node[anchor=north] at (1,1) {$5$};
		\node[anchor=north] at (2,1) {$2$};
		\node[anchor=north] at (1,2) {$2$};
		\node[anchor=north] at (3,3) {$\alpha$};
		\end{tikzpicture}
		\hfill
		\begin{tikzpicture}
		\foreach \x in {0,1,2}
		\foreach \y in {0,1,2}
		{
			\node[thickstyle](\x - \y) at (\x, \y){};
		}
		\node[thickstyle](3 - 0) at (3, 0){};
		\node[thickstyle](0 - 3) at (0, 3){};
		\draw[thick] (0, 0) -- (3, 0) -- (2, 2) -- (0, 3) -- cycle;
		\node [anchor=north] at (0, 0) {$1$};
		\node [anchor=north] at (1,0) {$3$};
		\node [anchor=north] at (2,0) {$3$};
		\node [anchor=north] at (3,0) {$1$};
		\node[anchor=east] at (0,1) {$3$};
		\node[anchor=east] at (0,2) {$3$};
		\node[anchor=east] at (0,3) {$1$};
		\node[anchor=south] at (2.2,1.9) {$1$};
		\node[anchor=north] at (1,1) {$6$};
		\node[anchor=north] at (2,1) {$3$};
		\node[anchor=north] at (1,2) {$4$};
		\node[anchor=north] at (3,3) {$\beta$};
		\end{tikzpicture}
		\hfill
		\begin{tikzpicture}
		\foreach \x in {0,1,2}
		\foreach \y in {0,1,2}
		{
			\node[thickstyle](\x - \y) at (\x, \y){};
		}
		\node[thickstyle](3 - 0) at (3, 0){};
		\node[thickstyle](0 - 3) at (0, 3){};
		\draw[thick] (0, 0) -- (3, 0) -- (2, 2) -- (0, 3) -- cycle;
		\node [anchor=north] at (0, 0) {$1$};
		\node [anchor=north] at (1,0) {$3$};
		\node [anchor=north] at (2,0) {$3$};
		\node [anchor=north] at (3,0) {$1$};
		\node[anchor=east] at (0,1) {$3$};
		\node[anchor=east] at (0,2) {$3$};
		\node[anchor=east] at (0,3) {$1$};
		\node[anchor=south] at (2.2,1.9) {$1$};
		\node[anchor=north] at (1,1) {$6$};
		\node[anchor=north] at (2,1) {$4$};
		\node[anchor=north] at (1,2) {$3$};
		\node[anchor=north] at (3,3) {$\gamma$};
		\end{tikzpicture}
		\caption{The three $0$-mutable polynomials $\alpha$, $\beta$ and $\gamma$ whose Newton polytope is the quadrilateral $F$ defined in \eqref{eq:quadrilateral}}
		\label{fig:114}
	\end{figure}
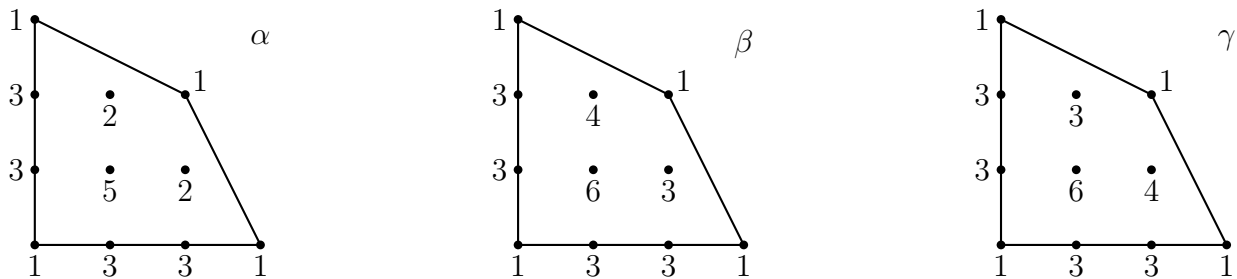
	
	Let us consider the following affine-linear functions $\LL = \ZZ^2 \to \ZZ$:
	\begin{align*}
		-m^1_{2,1} \colon (a,b) \mapsto b-1, \\
		-m^1_{3,1} \colon (a,b) \mapsto b-2, \\
		-m^1_{3,2} \colon (a,b) \mapsto 2b-1.
	\end{align*}
	The level sets of these three affine-linear functions are depicted in Figure~\ref{fig:level_sets}.
	\begin{figure}
		\centering
		\def\svgwidth{4cm}
		\input{Level_setsx.tex}
		\bigskip
		\medskip
		\caption{The level sets of the three affine-linear functions $\LL = \ZZ^2 \to \ZZ$ considered in Example~\ref{ex:quadrilateral}}
		\label{fig:level_sets}
	\end{figure}
	We now consider the mutation data $(-m^1_{2,1}, 1+x)$, $(-m^1_{3,2}, 1+x)$ and $(-m^1_{3,1}, 1+x)$.
	The polynomial $\alpha$ is mutable with respect to $(-m^1_{2,1}, 1+x)$ and to $(-m^1_{3,2}, 1+x)$ and
	\begin{gather*}
		\mut_{(-m^1_{2,1}, 1+x)} \alpha = \frac{1+x}{xy} + \frac{3+2x}{x} + \frac{y(3 + 2x + x^2)}{x} + \frac{y^2 (1+x)}{x}, \\
		\mut_{(-m^1_{3,2}, 1+x)} \alpha = \frac{1}{xy} + \frac{3+2x}{x} + \frac{y(3+5x + 3x^2 +x^3)}{x}  + \frac{y^2 (1+x)^3}{x},
	\end{gather*}
	but $\alpha$ is not mutable with respect to $(-m^1_{3,1}, 1+x)$.
	The polynomial $\beta$ is mutable with respect to all three mutation data $(-m^1_{2,1}, 1+x)$, $(-m^1_{3,2}, 1+x)$ and $(-m^1_{3,1}, 1+x)$, and the mutations are:
	\begin{gather*}
		\mut_{(-m^1_{2,1}, 1+x)} \beta = \frac{1+x}{xy} + \frac{3(1+x)}{x} + \frac{y(3 +x)(1+x)}{x} + \frac{y^2(1+x)}{x}, \\
		\mut_{(-m^1_{3,1}, 1+x)} \beta = \frac{1}{xy} + \frac{3}{x} + \frac{y(3 + x)}{x} + \frac{y^2}{x} = \frac{(1+y)^3 + xy^2}{xy}, \\
		\mut_{(-m^1_{3,2}, 1+x)} \beta = \frac{1}{xy} + \frac{3(1+x)}{x} + \frac{y(3 + x)(1+x)^2}{x} + \frac{y^2(1+x)^3}{x}.
	\end{gather*}
	The polynomial $\gamma$ is not mutable with respect to any of the mutation data $(-m^1_{2,1}, 1+x)$, $(-m^1_{3,2}, 1+x)$, $(-m^1_{3,1}, 1+x)$.
\end{exa}

\section{Properties of $0$-mutable polynomials}
\label{sec:properties-0-mutable}

\subsection{Some easy properties}

If $\LL$ is an affine lattice, $f\in \CC[\LL]$ and $F=\Newt f$, then we write
\[
f=\sum_{l\in F\cap \LL} a_l x^l \qquad \text{with} \qquad a_l\in \CC.
\]
For every subset $A\subseteq \LL_\RR$, we write
\[
f \vert_A=\sum_{l \in A\cap \LL} a_l x^l.
\]

\begin{lem}
	\label{lem:easy_properties}
	\begin{enumerate}[(1)]
		\item (Non-negativity and integrality) If $f$ is $0$-mutable, then every coefficient of $f$ is a non-negative integer.
		\item (Boundary terms) If $f\in \CC[\LL]$ is $0$-mutable and $F\leq \Newt f$ is a
		face, then $f \vert_F$ is $0$-mutable.
	\end{enumerate}
\end{lem}

\begin{proof}[Sketch of proof]
	(1) is obvious due to the recursive definition of 0-mutable polynomials. Also (2) is easy because it is enough to observe  that mutations and products behave well with respect to restriction to faces of the Newton polytope.
\end{proof}

\begin{rem}
	A $0$-mutable polynomial may have a zero coefficient at a lattice point of its Newton polytope, e.g.\ $(1+x)(1+xy^2)$.
\end{rem}

\subsection{Rigid maximally mutable polynomials}
\label{sec:rigid-maxim-mutable}

From now on we focus on the two-variable case $r=2$. In what follows,
we give two equivalent characterizations of $0$-mutable polynomials,
one geometric in terms of the associated cluster variety and one
combinatorial in terms of rigid maximally mutable
polynomials.\footnote{The concept of rigid maximally mutable is due to
	Al Kasprzyk \cite{KT}. We thank him for allowing us to include his
	definition here.}

Let $\LL$ be an affine lattice of rank $2$.
For a Laurent polynomial $f\in \CC[\LL]$, we set
\[
\seeds(f) = \Bigl\{ \text{mutation data $s=(\varphi, h)$} \ \Bigl\vert \  \text{$f$ is $s$-mutable}\Bigr\}.
\]
Conversely, if $\seeds$ is a set of mutation data, we denote by
\[
L(\seeds)=\Bigl\{f \in \CC[\LL]\ \Big\vert\  \text{$\forall s \in \seeds$, $f$ is $s$-mutable}\Bigr\}
\]
the vector space of Laurent polynomials $f$ that are $s$-mutable for
all the mutation data $s\in \seeds$. For every polynomial $f \in \CC[\LL]$, it is clear that $f \in L(\seeds(f))$.

\begin{dfn}[Kasprzyk]
	\label{dfn:rigidMMLP}
	Let $\LL$ be an affine lattice of rank $2$, and $f\in \CC[\LL]$.
	\begin{enumerate}[(i)]
		\item If $f=f_1f_2$ is the product of normalized polynomials $f_1$,
		$f_2$, then $f$ is rigid maximally mutable if both factors
		$f_1, f_2$ are rigid maximally mutable;
		\item If $f$ is normalized and irreducible, then $f$ is
		rigid maximally mutable if
		\[
		L \left( \seeds(f) \right) = \{\lambda f \mid \lambda \in \CC \}.
		\]
		
	\end{enumerate}
\end{dfn}

\subsection{Cluster varieties}

\begin{dfn}
	\label{dfn:CYpairANDcluster_var}
	A \emph{Calabi--Yau (CY) pair} is a pair $(Y, \omega)$ of an
	$n$-dimensional quasiprojective normal variety $Y$ and a degree $n$ rational
	differential $\omega \in \Omega^n_{k(X)}$, such that
	\[
	D =-\divi_Y \omega \geq 0
	\]
	is an effective reduced Cartier divisor on $Y$.\footnote{Like most people,
		we mostly work with the pair $(Y, D)$ and omit explicit reference
		to $\omega$.} 
	
	A \emph{torus chart} on $(Y,\omega)$ is an open embedding 
	\[
	j \colon (\CC^\times)^n \hookrightarrow Y\setminus D
	\qquad
	\text{such that}
	\qquad
	j^\star (\omega) = \frac{1}{(2\pi \rmi)^n} \frac{\rmd x_1}{x_1}\wedge \cdots \wedge \frac{\rmd x_n}{x_n}.
	\]
	
	A \emph{cluster variety} is an $n$-dimensional CY pair $(Y, \omega)$ that
	has a torus chart.
\end{dfn}

In our situation, the pair $(Y,D)$ will always be log smooth. 

\subsection{The cluster surface}
We construct a cluster surface from a
lattice polygon. 

Let $\LL$ be an affine lattice of rank $2$.
There is a canonical bijection
between the set of lattice polygons $F\subset \LL_\RR$ up to translation and
the set of pairs $(\Xbar, A)$ of a projective toric surface $\Xbar$ and an ample
line bundle $A$ on $\Xbar$: the torus in question is $\Spec \CC[\LL_0]$;
the fan of the surface $\Xbar$ is the normal fan of $F$, and it all works
out such that there is a natural $1$-to-$1$ correspondence between
$F\cap \LL$ and a basis of $H^0(\Xbar,A)$.

Fix a lattice polygon $F$ in $\LL$ and consider the corresponding
polarised toric surface $(\Xbar, A)$.
Denote by $\Bbar$ the toric boundary of $\Xbar$.
For each edge $E \leq F$,
let $\ell(E)$ be the lattice length of $E$ and let $\Bbar_E$ be the prime component
of $\Bbar$ corresponding to $E$;
we have that $\Bbar_E$ is isomorphic to $\PP^1$
and the line bundle $A \vert_{\Bbar_E}$ has degree $\ell(E)$.
Denote by $x_E \in \Bbar_E$ the point $[1:-1] \in \PP^1$.

For all edges $E \leq F$, blow up $\ell(E)$ times above $x_E$ in the proper
transform of $\Bbar_E$ and denote by
\[
p\colon (X,B) \longrightarrow (\Xbar, \Bbar)
\]
the resulting
surface, where $B \subset X$ is the proper transform of the toric
boundary $\Bbar = \sum_{E \leq F} \Bbar_E$.
We call the pair $(X,B)$ the \emph{cluster
	surface} associated to the lattice polygon $F$.

\begin{thm}
	\label{thm:equivalent_defns}
	Let $\LL$ be a rank-$2$ lattice. The following are equivalent for a
	normalized Laurent polynomial $f\in \CC [\LL]$:
	\begin{enumerate}[(1)]
		\item $f$ is $0$-mutable;
		\item $f$ is rigid maximally mutable;
		\item  Let
		$p \colon (X,B) \to (\Xbar, \Bbar)$ be the cluster surface
		associated to the polygon $F=\Newt f$.
		Denote by $Z \subset \Xbar$ the divisor of zeros of $f$ and by
		$Z^\prime \subset X$ the proper transform of $Z$. Every irreducible
		component $\Gamma \subset Z^\prime$ is a smooth rational curve
		with self-intersection $\Gamma^2=-2$. (Necessarily then $B \cdot \Gamma
		=0$ hence $\Gamma$ is disjoint from the boundary $B$.) 
	\end{enumerate}
\end{thm}

\begin{rem}
	\label{rem:4}
	The support of $Z^\prime$ is not necessarily a normal crossing
	divisor. The irreducible components need not meet transversally, and
	$\geq 3$ of them may meet at a point.
\end{rem}

\begin{proof}[Sketch of proof] In proving all equivalences we may and
	will assume that the polynomial $f$ is irreducible hence $Z$ is
	reduced and irreducible.
	
	The proof uses the following ingredients, which we state without
	further discussion or proof:
	\begin{enumerate}[(i)]
		\item To give a torus chart $j\colon \CC^{\times\, 2} \hookrightarrow
		X \setminus B$ in $X$ is the same as to give a
		\emph{toric model} of $(X,B)$, that is a projective morphism
		$q\colon (X, B) \to (X^\prime, B^\prime)$ where $(X^\prime,
		B^\prime)$ is a toric pair and $q$ maps
		$j(\CC^{\times \, 2})$ isomorphically to the torus $X^\prime \setminus
		B^\prime$; 
		\item The work of Blanc~\cite{MR3080816} implies that any two torus charts in $X$ are connected
		by a sequence of mutations between torus charts in $X$;
		\item A set $\seeds$ of mutation data specifies a line bundle $\cL(\seeds)$ on $X$
		such that $H^0\bigl(X,\cL(\seeds)\bigr)=L(\seeds)$ and, conversely, every line bundle on $X$
		is isomorphic to a line bundle of the form $\cL(\seeds)$. 
	\end{enumerate}
	
	Let us show first that (1) implies (3). To say that an irreducible
	polynomial $f$ is $0$-mutable is to say that there exists a sequence
	of mutations that mutates $f$ to the constant polynomial $\mathbf{1}$.
	This sequence of mutations constructs a new torus chart
	$j_1\colon \TT = \CC^{\times\, 2} \hookrightarrow X\setminus B$ such that the proper
	transform $Z^\prime$ -- which is, by assumption, irreducible -- is
	disjoint from $j_1 (\TT)$. This new toric chart gives a new toric model
	$p_1 \colon (X,B) \to (X_1, B_1)$ that maps $j_1 (\TT)$
	isomorphically to the torus $X_1 \setminus B_1$ and
	hence contracts $Z^\prime$ to a boundary point. $Z^\prime$ is not a $-1$-curve, because those
	are all $p$-exceptional, hence $Z^\prime$ is a $-2$-curve.
	
	To show that (3) implies (1), by Lemma~\ref{lem:clustery} below, there
	is a new toric model $p_1\colon (X, B) \to (X_1,B_1)$ that contracts
	$Z^\prime$ to a point in the boundary. The new toric model then gives a
	new torus chart $j_1\colon \TT \hookrightarrow X$ such that $Z^\prime $ is
	disjoint from $j_1(\TT)$. By Blanc the induced birational map of tori
	\[
	j_1^{-1}j \colon \TT \dasharrow \TT
	\]
	is a composition of mutations that mutates $f$ to the constant
	polynomial. 
	
	Let us now show that (3) implies (2). By some tautology, $Z^\prime$ is the
	zero divisor of the section $f$ of the line bundle $\cL(\seeds)$ on $X$ specified by
	the set of mutation data $\seeds=\seeds(f)$, and $H^0(X, Z^\prime)=L
	(\seeds)$. Since $Z^\prime$ is a $-2$-curve, $L(\seeds)$ is
	$1$-dimensional, which is to say that $f$ is rigid maximally mutable.
	
	Finally we show that (2) implies (3). Denote by $\cL=\cL(\seeds)$
	the line bundle on $X$ specified by the set of mutations $\seeds=\seeds(f)$, so
	that $Z^\prime$ is the zero-locus of a section of $\cL$. Note that:
	\[
	h^2(X, \cL)=h^0(X, K_X-Z^\prime)=h^0(Y, - \Bbar -Z^\prime)=0
	\]
	Riemann--Roch and the fact that $f$ is rigid give:
	\[
	1=h^0(X, Z^\prime)=h^0(X, \cL)\geq \chi(X, \cL)=1+ \frac{1}{2}(Z^{\prime
		\,2}+Z^\prime \Bbar)
	\]
	and hence, because $Z^\prime \Bbar \geq 0$, we conclude that $Z^{\prime \,
		2}\leq 0$ and:
	\[
	2p_a(Z^\prime)-2=\deg \omega_{Z^\prime} = Z^{\prime \, 2} - Z^\prime  \Bbar \leq 0
	\]
	so either:
	\begin{enumerate}[(i)]
		\item $\deg \omega_{Z^\prime}<0$ and then $Z^\prime$ is a smooth
		rational curve, and then as above $Z^\prime$ is not a $-1$ curve therefore it is a
		$-2$-curve, or
		\item $\omega_{Z^\prime}=\cO_{Z^\prime}$ and $Z^{\prime\, 2}=Z^\prime
		\Bbar=0$. It follows that $Z^\prime$ is actually disjoint from $\Bbar$ and
		$\cO_{Z^\prime}(Z^\prime) = \cO_{Z^\prime}$.
		The homomorphism
		\[
		H^0(X, Z^\prime) \to H^0\left(Z^\prime, \cO_{Z^\prime}(Z^\prime) \right)=\CC
		\]
		is surjective, hence actually $h^0(X, Z^\prime)=2$, a
		contradiction. 
	\end{enumerate}
	This means that we must be in case (i) where $Z^\prime$
	is a $-2$-curve. 
\end{proof}

\begin{lem}
	\label{lem:clustery}
	Let $(Y,D)$ be a cluster surface, and
	$Z^\prime \subset Y$ an interior $-2$-curve. Then there is a
	toric model $q\colon (Y,D) \to (X^\prime, B^\prime)$ that contracts
	$Z^\prime$. 
\end{lem}

\begin{proof}[Sketch of proof]
	First contract $Z^\prime$ to an interior node and then run a
	MMP. There is a small number of cases to discuss depending on how
	the MMP terminates.
\end{proof}

\section{Conjecture~B}
\label{sec:intr-mirr-symm}
In this section we state a new conjecture --- Conjecture~B --- which
implies, see Remark~\ref{rem:conjBimpliesmapdefined}, the existence of
the map $\kappa\colon \mathfrak{B}\to \mathfrak{A}$ of
Conjecture~A. In the strong form stated in
Remark~\ref{rem:conjBStrongForm}, together with Conjecture~A,
Conjecture~B asserts that the smoothing components of
$\Defo (V, \divV)$ are themselves smooth, and computes their tangent
space explicitly as a representation of the torus.  We conclude by
explaining how the two Conjectures~A and~B originate in mirror
symmetry. This last discussion is central to how we arrived at the
formulation of the conjectures, but it is not logically necessary for
making sense of their statement. We work with the version of mirror
symmetry put forward in~\cite{MR3415066} and~\cite{gs_intrinsic_slc,
	1909.07649}.\footnote{We thank Paul Hacking for several helpful
	discussions on mirror symmetry and for correcting earlier drafts of
	this section. We are of course responsible for the mistakes that are
	left.}

\medskip

For the remainder of this section fix a lattice polygon $F$ and
denote, as usual, by $V$ the corresponding Gorenstein toric affine
\mbox{$3$-fold} with toric boundary $\partial V$.

In this section we always work with the space $\Defo
(V, \divV)$. Also fix a $0$-mutable polynomial $f \in \CC[\LL]$ with
$\Newt f = F$. 
Conjecture~B associates to $f$ a
$\TT$-equivariant family $(\cU_f, \cD_f) \to \moduli{f}$ of deformations of the pair $(V, \partial V)$.

In the last part of this section we construct from $f$ a $3$-dimensional cluster variety
$(W,D)$, conjecturally the \emph{mirror} of $\moduli{f}$, and hint at an explicit
conjectural construction of the family $\cU_f \to \moduli{f}$ from
the degree-$0$ quantum log cohomology of $(W, D)$.

\medskip

Denote by $\sigma \subset N_\RR$ the cone over $F$ at height $1$.
As
usual, $\usp \in M = \Hom_\ZZ (N,\ZZ)$ denotes the Gorenstein degree,
so $F=\sigma \cap \LL$ where $\LL=(u=1)$.

\subsection{Statement of Conjecture~B}
\label{sec:conjectures}

As in Sec.~\ref{sec:rigid-maxim-mutable}, denote by $\seeds(f)$ the set
of mutation data of $f$. Recall that an element of $\seeds(f)$ is a
pair $(\varphi, h)$ consisting of an affine function
$\varphi \in \Aff (\LL, \ZZ)$ and a Laurent polynomial
$h\in \CC[\Ker \varphi_0]$. Using the restriction isomorphism
$M \simeq \Aff(\LL, \ZZ)$, when it
suits us we view a mutation datum $(\varphi, h)$ as a pair of an element $\varphi \in M$ and a
polynomial $h\in \CC[\LL_0\cap \Ker \varphi] \subseteq \CC[N]$.

The most useful mutation data are those where $\varphi$ is strictly
negative somewhere on $F$, and then the minimum of $\varphi$ on $F$ is
achieved on an edge $E\leq F$.\footnote{Recall that in
	Definition~\ref{dfn:mutation} we explicitly assume that $\varphi$ is not constant on $F$.} In the
present discussion we want to focus on these mutation data:
\[
\seeds_{-}(f)= \left\{ (\varphi, h)\in \seeds(f) \mid \text{for some edge $E\leq
	F$, $\varphi|_E$ is constant and $<0$} \right\}.
\]

We want to define a seed $\widetilde{\seeds}(f)$ on $N$, that is, a set of
pairs $(\varphi,h)$ of a character $\varphi \in M$ and a Laurent
polynomial $h\in \CC[\Ker \varphi ]$. The seed we want is
\[
\widetilde{\seeds}(f)=\seeds_{-}(f)
\cup \Bigl\{ (-ku,h)\ \Big\vert\ \text{$h$ is a prime factor of $f$ of multiplicity
	$k$}\Bigr\}.
\]

\begin{conB}
	In this statement, if $U$ is a representation of the torus $\TT = \Spec \CC[M]$ and $m\in M$
	is a character of $\TT$, we denote by $U(m)$ the direct summand of $U$
	on which $\TT$ acts with pure weight $m$. 
	
	Let $F\subset \LL$ be a lattice polygon, $V$ the
	corresponding Gorenstein toric affine $3$-fold, and $f\in \CC[\LL]$
	a $0$-mutable polynomial with $\Newt f = F$. 
	For every integer $k\geq 1$, denote by $n_k$ the number of prime
	factors of $f$ of multiplicity $\geq k$.
	
	Then there is a $\TT$-invariant submanifold $\moduli{f} \subset \Defo (V, \divV)$
	such that
	\[
	\dim T_{0} \moduli{f} (m) =
	\begin{cases}
	1 & \; \text{if $m \not \in \langle-u\rangle_+$ and $\exists$ $(\varphi, h)\in
		\widetilde{\seeds}(f)$ such that $m=\varphi$},\\
	n_k & \; \text{if $m=-ku$ for some integer $k \geq 1$},\\
	0 & \; \text{otherwise},
	\end{cases}
	\]
	and the general fibre of the
	family over $\moduli{f}$ is a pair consisting of a smooth variety and a smooth divisor. 
\end{conB}

\begin{rem}
	\label{rem:conjBimpliesmapdefined}
	If Conjecture~B holds then by openness of versality a general point
	of $\moduli{f}$ lies in a unique component of $\Defo (V, \divV)$ and
	this gives the map $\kappa\colon \mathfrak{B}\to \mathfrak{A}$ in
	the statement of Conjecture~A (see also
	Remark~\ref{rem:deformations_variety_or_pair}).
\end{rem}

\begin{rem}[Strong form of Conjecture~B]
	\label{rem:conjBStrongForm}
	A strong form of Conjecture~B states that the families
	$\moduli{f}$ are precisely the smoothing components of
	$\Defo (V,\divV)$.
\end{rem}

\begin{exa}[Example~\ref{ex:quadrilateral}
	continued] \label{ex:quadrilateral_conjecture} Consider the quadrilateral $F$ in $\LL = \ZZ^2$ defined in \eqref{eq:quadrilateral} and the three $0$-mutable polynomials $\alpha$, $\beta$ and $\gamma$ with Newton polytope $F$.
	Set $N = \LL \oplus \ZZ = \ZZ^3$ and consider the following linear functionals in $M = \Hom_\ZZ(N,\ZZ) = \ZZ^3$:
	\begin{align*}
		-m^1_{2,1} &= (0,1,-1)   &  -m^2_{2,1} &= (1,0,-1) \\
		-m^1_{3,1} &= (0,1,-2)    &  -m^2_{3,1} &= (1,0,-2) \\
		-m^1_{3,2} &= (0,2,-1)    &  -m^2_{3,2} &= (2,0,-1)
	\end{align*}
	and $-u = (0,0,-1)$.
	The names $-m^1_{2,1}$, $-m^1_{3,1}$ and $-m^1_{3,2}$ are compatible with the affine-linear functions considered in Example~\ref{ex:quadrilateral} via the restriction isomorphism $M \simeq \Aff(\LL, \ZZ)$.
	Set $x = x^{(1,0,0)} \in \CC[N]$ and $y = x^{(0,1,0)} \in \CC[N]$. Then we have:
	\begin{gather*}
		\widetilde{\seeds}(\alpha) \supseteq \left\{ (-u, 1+x+2y + y^2) , (-u, 1+y+2x+x^2)  \right\}, \\
		\widetilde{\seeds}(\alpha) \supseteq \left\{ (-m^1_{2,1}, 1+x), (-m^1_{3,2}, 1+x), (-m^2_{2,1}, 1+y), (-m^2_{3,2}, 1+y) \right\}, \\
		\widetilde{\seeds}(\beta) \supseteq \left\{(-u, \beta), (-m^1_{2,1}, 1+x), (-m^1_{3,1}, 1+x), (-m^1_{3,2}, 1+x) \right\}, \\
		\widetilde{\seeds}(\gamma) \supseteq \left\{(-u, \gamma), (-m^2_{2,1}, 1+y), (-m^2_{3,1}, 1+y), (-m^2_{3,2}, 1+y) \right\}.
	\end{gather*}
	
	Let $V$ be the Gorenstein toric affine $3$-fold associated to $F$.
	Conjecture~B states that there are three submanifolds $\moduli{\alpha}$, $\moduli{\beta}$ and $\moduli{\gamma}$ of $\Defo (V, \partial V)$ such that the dimensions of $T_0 \moduli{\alpha}(m)$, $T_0 \moduli{\beta}(m)$ and $T_0 \moduli{\gamma}(m)$ for $m \in \{-u, -m^1_{2,1}, -m^1_{3,1}, -m^1_{3,2}, -m^2_{2,1}, -m^2_{3,1}, -m^2_{3,2} \}$ are written down in the table below. 
	\[
	\begin{array}{cccccccc}
	& -u & -m^1_{2,1} & -m^1_{3,1} & -m^1_{3,2} & -m^2_{2,1} & -m^2_{3,1} & -m^2_{3,2} \\
	\hline
	\dim T_0 \moduli{\alpha}(m)     & 2 & 1 & 0 & 1 & 1 & 0 & 1 \\
	\dim T_0 \moduli{\beta}(m)       & 1 & 1 & 1 & 1 & 0 & 0 & 0 \\
	\dim T_0 \moduli{\gamma}(m)  & 1 & 0 & 0 & 0 & 1 & 1 & 1 
	\end{array}
	\]
\end{exa}

\subsection{Mirror symmetry interpretation}
\label{sec:mirr-symm-interpr}

Denote by $\cU_f \to \moduli{f}$ the deformation family of $V$ induced by the composition
$\moduli{f} \into \Defo(V, \divV) \twoheadrightarrow \Defo V$. We sketch a construction of
$\cU_f$ in the spirit of intrinsic mirror symmetry
\cite{gs_intrinsic_slc, 1909.07649}.

\medskip 

Let $\sigma^\vee \subset M_\RR$ be the dual cone, and let
$s_1, \dots, s_r$ be the primitive generators of the rays of
$\sigma^\vee$. Denote by $\Wb$ the toric variety --- for the dual
torus $\TT_M=\Spec \CC[N]$ --- constructed from the fan consisting of
the cones $\{0\}$, $\langle -u\rangle_+$, the $\langle s_j\rangle_+$, and the
two-dimensional cones
\[
\langle -u, s_j\rangle_+
\]
(for $j=1, \dots, r$), and let $\Db \subset \Wb$ be the toric
boundary. 

\medskip

Now, the set of edges $E\leq F$ is in $1$-to-$1$ correspondence with
the set $\{ s_1, \dots, s_r \}$, where $E$
corresponds to $s_j$ if $s_j|_E=0$. If
$(\varphi, h) \in \widetilde{\seeds}(f)$ and
$\varphi \not \in \langle -u\rangle_+$, then there is a unique $j$
such that $\varphi$ is in the cone $\langle -u, s_j\rangle_+$ spanned
by $-u$ and $s_j$.

Let
\[(\Wbar, \Dbar) \longrightarrow (\Wb, \Db)\]
be the toric variety obtained by
adding the rays $\langle \varphi \rangle_+ \subset M_\RR$ whenever
$(\varphi, h)\in \widetilde{\seeds}(f)$, and (infinitely many)
two-dimensional cones subdividing the cones
$\langle -u, s_j\rangle_+$. Note that $\Wbar$ is not quasi-compact and
not proper.

\medskip

Finally, we construct a projective morphism \[ (W, D) \longrightarrow (\Wbar,
\Dbar)\] by a sequence of blowups. 

In what follows for all $(\varphi, h)\in \widetilde{\seeds}(f)$ we
denote by $D_{\langle \varphi \rangle_+}\subset \Wbar$ the
corresponding boundary component, and set 
\[
Z_h=(h=0)\subset D_{\langle \varphi \rangle_+}.
\]

The following simple remarks will be helpful in describing the construction.
\begin{enumerate}
	\item[(a)] For all positive integers $k$, $(k\varphi, h)\in \seeds_{-}(f)$
	if and  only if $(\varphi, h^k)\in \seeds_{-}(f)$.
	\item[(b)] By construction, if $(\varphi, h)\in \seeds_{-}(f)$, then one has $h=(1+x^e)^k$
	(up to translation) for some positive integer $k$, 
	where $e\in M$ is a primitive lattice vector along the edge $E\leq F$ where
	$\varphi$ achieves its minimum.
\end{enumerate}

Let $R\subset M_\RR$ be a ray of the fan of $\Wbar$ other than
$\langle -u\rangle_+$, and let $\varphi
\in M$ be the primitive generator of $R$. It follows from the
remarks just made that there is a largest positive integer $k_R$
such that $(k_R \varphi, 1+x^e)\in \seeds_{-}(f)$.

\medskip

Our mirror $W$ is obtained from $\Wbar$ by:
\begin{enumerate}
	\item First, as $R$ runs through all the rays of the fan of $\Wbar$
	other than $\langle-u\rangle_+$ in some order, blow up $k_R$ times
	above $(1+x^e=0)$ in the proper transform of $D_R$. It can be seen, and it
	is a nontrivial fact, that after doing all these blowups the $Z_h$
	in the proper transforms of the $D_{\langle -u\rangle_+}$ are
	smooth;
	\item Subsequently, if $f=\prod h^{k(h)}$ where the $h$ are
	irreducible, blow up in any order $k(h)$ times above $Z_h$ in the
	proper transform of $D_{\langle -u\rangle_+}$.
\end{enumerate}

The resulting CY pair $(W, D)$ is log smooth, because we have blown up
a sequence of smooth centres. This $(W, D)$ is the \emph{mirror} of
$\moduli{f}$.

One can see that the Mori cone $\NE(W/\Wbar)$ is simplicial, and hence
there is an identification:
\[
\moduli{f} = \Spec \CC[\NE(W/\Wbar)]
\]
Mirror symmetry suggests --- modulo issues with
infinite-dimensionality ---
that the ring $QH^0_{\log} (W,D;\CC[\NE(W/\Wbar)])$ has a natural filtration and that
one recovers the universal family of pairs $(\cU_f, \cD_f)\to \cM_f$ from this ring out of the
Rees construction.

\begin{rem}
	In the context of Conjecture~B, it would be very nice to work out an
	interpretation of the coefficients of the $0$-mutable polynomial $f$
	as counting certain holomorphic disks on the general fibre of the
	family $\cU_f\to \moduli{f}$.
\end{rem}

\section{Evidence}
\label{sec:evidence-1}

We have already remarked that the variety $V$ of Example~\ref{exa:conePP123} confirms Conjecture~A. Here we collect some further evidence.
Section~\ref{sec:evidence-matej} is a study of $\Defo V$ where $V$ is the variety of Example~\ref{ex:quadrilateral} and Example~\ref{ex:quadrilateral_conjecture}.

\subsection{Isolated singularities}

Here we fix a lattice polygon $F$ with unit edges, i.e.\ edges with lattice length $1$.
Let $V$ be the Gorenstein toric affine \mbox{$3$-fold} associated to $F$; we have that $V$ has an isolated singularity.

Altmann~\cite{altmann_versal_deformation} proved that there is a 1-to-1 correspondence between the irreducible components of $\Defo V$ and the maximal Minkowski decompositions of $F$.
This restricts to a 1-to-1 correspondence between the smoothing components of $\Defo V$ and the Minkowski decompositions of $F$ with summands that are either unit segments or standard triangles. Here a standard triangle is a lattice triangle that is $\ZZ^2 \rtimes \mathrm{GL}_2(\ZZ)$-equivalent to $\hull{(0,0),(1,0),(0,1)}$.

On a polygon $F$ with unit edges, the $0$-mutable polynomials are exactly those that are associated to the Minkowski decompositions of $F$ into unit segments and standard triangles. This confirms Conjecture~A.

\subsection{Local complete intersections}

Nakajima \cite{nakajima} has characterised the affine toric varieties that are local complete intersection (lci for short). These are Gorenstein toric affine varieties associated to certain lattice polytopes called Nakajima polytopes. We refer the reader to \cite[Lemma~2.7]{toric_lci_are_resolvable} for an inductive characterisation of Nakajima polytopes. From this characterisation it is very easy to  see that every Nakajima polygon is affine equivalent to 
\begin{equation} \label{eq:nakajima}
F_{a,b,c} = \hull{(0,0),(a,0),(0,b),(a,b+ac)}
\end{equation}
in the lattice $\ZZ^2$, for some non-negative integers $a,b,c$ such that $a \geq 1$ and $ b + c \geq 1$.
It is easy to show that the Gorenstein toric affine \mbox{$3$-fold} associated to the polygon $F_{a,b,c}$ is
\[
V_{a,b,c} = \Spec \CC[x_1,x_2,x_3,x_4,x_5]/(x_1 x_2 - x_4^c x_5^b, \ x_3 x_4 - x_5^a).
\]

There is a unique $0$-mutable polynomial on $F_{a,b,c}$: this is associated to the unique Minkowski decomposition of $F_{a,b,c}$ into $a$ copies of the triangle $F_{1,0,c} = \hull{(0,0),(1,0),(0,c)}$ and $b$ copies of the segment $\hull{(0,0),(0,1)}$. On the other hand, as $V_{a,b,c}$ is lci, we have that $V_{a,b,c}$ is unobstructed and smoothable, therefore there is a unique smoothing component in the miniversal deformation space of $V_{a,b,c}$. This confirms Conjecture~A.

\section{A worked example}
\addtocontents{toc}{\setcounter{tocdepth}{1}} 
\label{sec:evidence-matej}

We explicitly compute the smoothing components of the miniversal
deformation space of the Gorenstein toric affine \mbox{$3$-fold} $V$
associated to the polygon $F$ of Example~\ref{ex:quadrilateral}
(continued in Example~\ref{ex:quadrilateral_conjecture}).  We saw that
there exist exactly three $0$-mutable polynomials with Newton polytope
$F$: $\alpha$, $\beta$ and $\gamma$.  We explicitly compute the
miniversal deformation space of $V$ and see that it has three
irreducible components, all of which are smoothing components. This
confirms our conjectures.

\subsection{The equations of $V$}
We consider the quadrilateral
\begin{equation} \label{eq:quadrilateral_last}
F = \hull{\vectortwo{-1}{-1}, \vectortwo{2}{-1}, \vectortwo{1}{1}, \vectortwo{-1}{2}}
\end{equation}
in the lattice $\LL = \ZZ^2$.
We consider the cone $\sigma$ obtained by placing $F$ at height $1$,
i.e.\ $\sigma$ is the cone generated by
\begin{equation*}
	a_1=\vectorthree{-1}{-1}{1}  \qquad a_2=\vectorthree{2}{-1}{1}  \qquad a_3=\vectorthree{1}{1}{1}  \qquad a_4=\vectorthree{-1}{2}{1}
\end{equation*}
in the lattice $N = \LL \oplus \ZZ = \ZZ^3$, and the corresponding
Gorenstein toric affine \mbox{$3$-fold}
$V = \Spec \CC[\sigma^\vee \cap M]$, where $\sigma^\vee$ is the dual cone of $\sigma$ in the dual lattice
$M = \Hom_\ZZ(N,\ZZ) \simeq \ZZ^3$.

The Gorenstein degree is
\[
\usp = (0,0,1) \in M.
\]
The primitive generators of the rays of the dual cone $\sigma^\vee \subseteq M_\RR$ are the vectors
\begin{align*}
	s_1 = (0,1,1), \quad s_2 = (1,0,1), \quad s_3 = (-1,2,3), \quad s_4 = (-2,-1,3)
\end{align*}
which are orthogonal to the $4$ edges of $F$.
The Hilbert basis of the monoid $\sigma^\vee \cap M$ is the set of the vectors
\begin{align*}
	\usp, \ s_1, \ z_2 = (-1,0,2),\  s_4, \ z_3 = (-1,-1,2), \ s_3, \ z_4 = (0,-1,2), \ s_2.
\end{align*}
Notice that these are the Gorenstein degree $\usp$ and
certain lattice vectors on the boundary
of $\sigma^\vee$.
The elements of the Hilbert basis of $\sigma^\vee \cap M$ are depicted in Figure~\ref{fig:slices}.
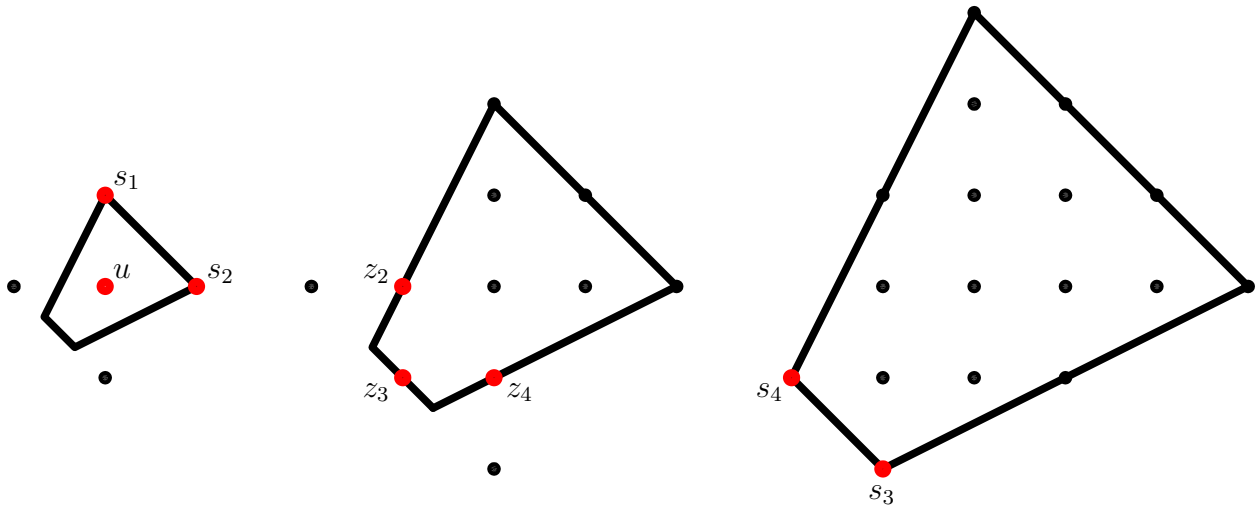
\begin{figure}
	\centering
	\def\svgwidth{\textwidth}
\begingroup%
  \makeatletter%
  \providecommand\color[2][]{%
    \errmessage{(Inkscape) Color is used for the text in Inkscape, but the package 'color.sty' is not loaded}%
    \renewcommand\color[2][]{}%
  }%
  \providecommand\transparent[1]{%
    \errmessage{(Inkscape) Transparency is used (non-zero) for the text in Inkscape, but the package 'transparent.sty' is not loaded}%
    \renewcommand\transparent[1]{}%
  }%
  \providecommand\rotatebox[2]{#2}%
  \newcommand*\fsize{\dimexpr\f@size pt\relax}%
  \newcommand*\lineheight[1]{\fontsize{\fsize}{#1\fsize}\selectfont}%
  \ifx\svgwidth\undefined%
    \setlength{\unitlength}{726.36021683bp}%
    \ifx\svgscale\undefined%
      \relax%
    \else%
      \setlength{\unitlength}{\unitlength * \real{\svgscale}}%
    \fi%
  \else%
    \setlength{\unitlength}{\svgwidth}%
  \fi%
  \global\let\svgwidth\undefined%
  \global\let\svgscale\undefined%
  \makeatother%
  \begin{picture}(1,0.37832466)%
    \lineheight{1}%
    \setlength\tabcolsep{0pt}%
    \put(0,0){\includegraphics[width=\unitlength,page=1]{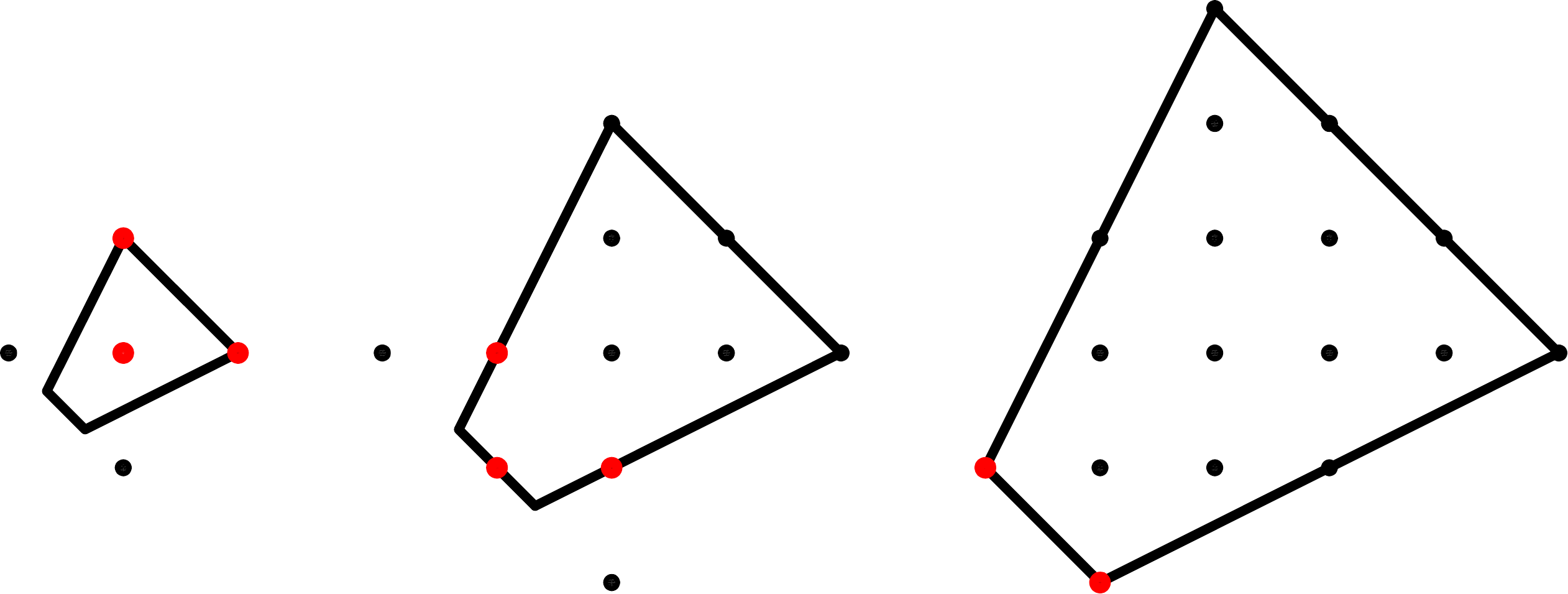}}%
    \put(0.085,0.16){\makebox(0,0)[lt]{\lineheight{1.25}\smash{\begin{tabular}[t]{l}$u$\end{tabular}}}}%
    \put(0.16,0.16){\makebox(0,0)[lt]{\lineheight{1.25}\smash{\begin{tabular}[t]{l}$s_2$\end{tabular}}}}%
    \put(0.085,0.235){\makebox(0,0)[lt]{\lineheight{1.25}\smash{\begin{tabular}[t]{l}$s_1$\end{tabular}}}}%
    \put(0.285,0.16){\makebox(0,0)[lt]{\lineheight{1.25}\smash{\begin{tabular}[t]{l}$z_2$\end{tabular}}}}%
    \put(0.285,0.065){\makebox(0,0)[lt]{\lineheight{1.25}\smash{\begin{tabular}[t]{l}$z_3$\end{tabular}}}}%
    \put(0.4,0.065){\makebox(0,0)[lt]{\lineheight{1.25}\smash{\begin{tabular}[t]{l}$z_4$\end{tabular}}}}%
    \put(0.6,0.065){\makebox(0,0)[lt]{\lineheight{1.25}\smash{\begin{tabular}[t]{l}$s_4$\end{tabular}}}}%
    \put(0.69,-0.017){\makebox(0,0)[lt]{\lineheight{1.25}\smash{\begin{tabular}[t]{l}$s_3$\end{tabular}}}}%
  \end{picture}%
\endgroup%

	\bigskip
	\caption{The intersections of the cone $\sigma^\vee$ with the planes $\RR^2 \times \{1\}$, $\RR^2 \times \{2\}$ and $\RR^2 \times \{3\}$ in $M_\RR = \RR^3$}
	\label{fig:slices}
\end{figure}

The elements of the Hilbert basis of $\sigma^\vee \cap M$ give a closed embedding of $V$ inside $\AA^8$ such that the ideal is generated by binomial equations. 
By using rolling factors formats (see \cite{stevens_rolling} and \cite[\S12]{stevens_book}), one can\footnote{We are obliged to the referee for suggesting these equations to us.} see that these equations are:
\begin{gather*}
	\mathrm{rank}
	\begin{pmatrix}
		x_{s_1} & x_{z_2} & x_\usp & x_{s_2} & x_{z_4}\\
		x_{z_2} & x_{s_4} & x_{z_3} & x_{z_4} & x_{s_3}
	\end{pmatrix}
	\leq 1, \\
	x_{s_4}x_{s_3}-x^3_{z_3}=0 \qquad \qquad x_{z_2} x_{s_3} - x_{z_3}^2 x_u = 0, \\
	x_{z_2}x_{z_4}-x_{z_3} x_\usp^2=0 \qquad \qquad  x_{s_1} x_{z_4} - x_\usp^3=0.
\end{gather*}

The singular locus of $V$ has two irreducible components of dimension
$1$: $V$ has generically transverse $A_2$-singularities along each of these.

\subsection{The tangent space}

We consider the tangent space to the deformation functor of $V$, i.e.\ $T^1_V = \Ext^1_{\cO_V}(\Omega_V^1, \cO_V)$.
This is a $\CC$-vector space with an $M$-grading.
For every $m \in M$ we denote by $T^1_V(-m)$ the graded component of $T^1_V$ of degree $-m$.

\begin{lem} \label{lem:t1}
	We define $\degs := \{ (p,q) \in \ZZ^2 \mid 2 \leq p \leq 3, \ q \geq 1 \}$.
	For all $p, q \in \ZZ$ we set $m^1_{p,q} := p \usp - q s_1$ and
	$
	m^2_{p,q} := p \usp - q s_2$.
	
	Then
	\[
	\dim T^1_V(-m) =
	\begin{cases}
	1 &\text{if } m = \usp, \\
	1 &\text{if } m = m^1_{p,q} \text{ with } (p,q) \in \degs, \\
	1 &\text{if } m = m^2_{p,q} \text{ with } (p,q) \in \degs, \\
	0 &\text{otherwise.}
	\end{cases}
	\]
\end{lem}

\begin{proof}
	This is a direct consequence of \cite[Theorem~4.4]{altmann_one_parameter}.
\end{proof}

Some of the degrees of $T^1_V$ are depicted in Figure~\ref{fig:slab}.

\begin{figure}
	\centering
	\def\svgwidth{4cm}
	\input{Slabx.tex}
	\caption{Some degrees of $T^1_V$ in the plane $\RR \times \{0\} \times \RR \subseteq M_\RR = \RR^3$.}
	\label{fig:slab}
\end{figure}

\medskip

The base of the miniversal deformation of $V$ is the formal completion (or germ) at the origin
of a closed subscheme of the countable-dimensional affine space $T^1_V$.
We denote by $t_m$
the coordinate on the 1-dimensional $\CC$-vector space $T^1_V(-m)$,
when $m = \usp$ or $m \in \{m^1_{p,q}, m^2_{p,q} \}$ with
$(p,q) \in \degs$.
Since we want to understand the structure of $\Defo V$, we want to analyse the equations of $\Defo V \into T^1_V$
in the variables $t_u$, $t_{m^1_{p,q}}$ and $t_{m^2_{p,q}}$ for $(p,q) \in J$.

The first observation is that each homogeneous first order deformation of $V$ is unobstructed as we see in the following two remarks.

\begin{rem} \label{rem:Minkowski_sum_gives_deform} The 1-dimensional
	$\CC$-vector space $T^1_V(-\usp)$ gives a first order deformation of
	$V$, i.e.\ an infinitesimal deformation of $V$ over
	$\CC[t_\usp]/(t_\usp^2)$.  This deformation can be extended to an
	algebraic deformation of $V$ over $\CC[t_u]$ as follows
	(see \cite{altmann_minkowski_sums}).
	
	Consider the unique non-trivial Minkowski decomposition of $F$ (see Figure~\ref{fig:minkowski_dec_F}):
	\begin{equation} \label{eq:Minkowski_deco_of_F}
	F = \hull{\vectortwo{0}{0},\vectortwo{1}{0},\vectortwo{0}{2}} + \hull{\vectortwo{-1}{-1},\vectortwo{1}{-1},\vectortwo{-1}{0}}.
	\end{equation}
	Let $\widetilde{F}$ be the Cayley polytope associated to this Minkowski sum; $\widetilde{F}$ is a $3$-dimensional lattice polytope.
	Let $\tilde{\sigma}$ be the cone over $\widetilde{F}$ at height $1$, i.e.\ $\tilde{\sigma}$ is the $4$-dimensional cone generated by
	\begin{equation*}
		\vectorfour{0}{0}{1}{0}, \vectorfour{1}{0}{1}{0}, \vectorfour{0}{2}{1}{0}, \vectorfour{-1}{-1}{0}{1}, \vectorfour{1}{-1}{0}{1}, \vectorfour{-1}{0}{0}{1}.
	\end{equation*}
	Let $\widetilde{V}$ be the Gorenstein toric affine \mbox{$4$-fold} $\widetilde{V}$ associated to $\tilde{\sigma}$. Consider the difference of the two regular functions on $\widetilde{V}$ associated to the characters $(0,0,1,0)$ and $(0,0,0,1)$; if we consider this regular function on $\widetilde{V}$ as a morphism $\widetilde{V} \to \AA^1$, we obtain the following cartesian diagram
	\begin{equation*}
		\xymatrix{
			V \ar[r] \ar[d] & \widetilde{V} \ar[d] \\
			\{0 \} \ar[r] & \AA^1
		}
	\end{equation*}
	which gives the wanted $1$-parameter deformation of $V$.

	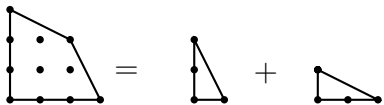
\begin{figure}
		\begin{tikzpicture}[scale=0.4]
		\draw[thick,  color=black]  
		(0,0) -- (3,0) -- (2,2) -- (0,3) -- cycle;
		\fill[thick,  color=black]
		(0,0) circle (3.5pt) (3,0) circle (3.5pt) (2,2) circle (3.5pt)
		(2,0) circle (3.5pt) (1,0) circle (3.5pt) (1,1) circle (3.5pt) (2,1) circle (3.5pt) (1,2) circle (3.5pt) (0,2) circle (3.5pt) (0,1) circle (3.5pt) (0,3) circle (3.5pt);
		\draw[thick,  color=black]
		(3, 0.9) node{$\hspace{2.2em}=\hspace{0.5em}$};
		\end{tikzpicture}
		\begin{tikzpicture}[scale=0.4]
		\draw[thick,  color=black]  
		(0,0) -- (1,0) -- (0,2) -- cycle;
		\fill[thick,  color=black]
		(0,0) circle (3.5pt) (1,0) circle (3.5pt) (0,2) circle (3.5pt) (0,1) circle (3.5pt);
		\draw[thick,  color=black]
		(1.5, 0.9) node{$\hspace{2.2em}+\hspace{0.5em}$};
		\end{tikzpicture}
		\begin{tikzpicture}[scale=0.4]
		\draw[thick,  color=black]  
		(0,0) -- (2,0) -- (0,1) -- cycle;
		\fill[thick,  color=black]
		(0,0) circle (3.5pt) (2,0) circle (3.5pt) (0,1) circle (3.5pt) (0,1) circle (3.5pt) (1,0) circle (3.5pt);
		\end{tikzpicture}
		\caption{The Minkowski decomposition \eqref{eq:Minkowski_deco_of_F} of the quadrilateral $F$ defined in \eqref{eq:quadrilateral_last}}
		\label{fig:minkowski_dec_F}
	\end{figure}
\end{rem}

\begin{rem}
	For every $m \in \{ m^1_{p,q}, m^2_{p,q} \}$ with $(p,q) \in \degs$,
	the first order deformation of $V$ corresponding to
	$T^1_V(-m) \simeq \CC$ can be extended to an algebraic deformation
	of $V$ over $\CC[t_m]$ thanks to
	\cite[Theorem~3.4]{altmann_one_parameter} (see also \cite{mavlyutov,
		petracci_mavlyutov}).
\end{rem}

\subsection{The obstruction space}

We now consider the obstruction space of the deformation functor of $V$:
$T^2_V = \Ext^2_{\cO_V}(\Omega^1_V, \cO_V)$. This is an $M$-graded
$\CC$-vector space. For all $m\in M$ we denote by $T^2_V(-m)$ the
direct summand of degree $-m$.

\begin{lem}\label{lem t2}
	If $m \in \{4\usp-s_1, \ 4\usp-s_2, \ 5\usp-s_1-s_2, \ 6\usp-s_1-s_2, \ 9\usp-2s_1-2s_2\}$, then $\dim T^2_V(-m) = 1$. Otherwise $\dim T^2_V(-m) = 0$.
\end{lem}

\begin{proof}
	This is a direct computation using formulae in \cite[Section 5]{altmann_andre_quillen_cohomology}.
\end{proof}

\begin{rem}
	It immediately follows from the computation in \cite[Section
	5]{altmann_andre_quillen_cohomology} that $T^2_V(-m)=0$ if there
	exists $a_i$ such that $\lan m,a_i \ran \leq 0$.
\end{rem}

\subsection{Verifying the conjectures}

Since $\dim T^2_V = 5$, the ideal of the closed embedding $\Defo V \into T^1_V$ has at most $5$ generators. We have the following:

\begin{pro} \label{pro:eq}
	The equations of the closed embedding $\Defo V \into T^1_V$ are
	\begin{align*}
		t_{m^1_{3,1}} t_\usp = 0, \\
		t_{m^2_{3,1}} t_\usp = 0,  \\
		t_{m^1_{3,1}} t_{m^2_{2,1}} + t_{m^1_{2,1}} t_{m^2_{3,1}}= 0,  \\
		t_{m^1_{3,1}}t_{m^2_{3,1}} = 0, \\
		t^2_{m^2_{3,1}} t_{m^1_{3,2}} - t^2_{m^1_{3,1}} t_{m^2_{3,2}} = 0.
	\end{align*}
	Moreover, $\Defo V$ is non-reduced and has exactly 3 irreducible components; their equations inside $T^1_V$ are:
	\begin{enumerate}
		\item $t_{m^1_{3,1}}=t_{m^2_{3,1}}=0$,
		\item $t_{\usp}=t_{m^1_{3,1}}=t_{m^1_{2,1}}=t_{m^1_{3,2}}=0$,
		\item $t_{\usp}=t_{m^2_{3,1}}=t_{m^2_{2,1}}=t_{m^2_{3,2}}=0$.
	\end{enumerate}
	Every irreducible component of $\Defo V$ is smooth and is a smoothing component.
\end{pro}

\begin{proof}
	The proof of the equations of $\Defo V \into T^1_V$ is postponed to
	the next section and relies on some computer calculations performed
	with Macaulay2. We now assume to know these equations.
	
	The fact that $\Defo V$ is non-reduced and has $3$ irreducible
	components $C_1$, $C_2$, $C_3$ with the equations given above can
	be checked by taking the primary decomposition of the ideal of
	$\Defo V \into T^1_V$.  For each $i=1,2,3$, from the equations of
	$C_i$ it is obvious that $C_i$ is smooth. We need to prove that
	$C_i$ is a smoothing component, i.e.\ the general fibre over $C_i$
	is smooth.
	
	The component $C_1$ contains the $1$-parameter deformation
	constructed in Remark~\ref{rem:Minkowski_sum_gives_deform}. The
	singular locus of the general fibre of this deformation has $2$
	connected components with everywhere transverse
	$A_2$-singularities; therefore the general fibre of this
	deformation is smoothable.
	
	In order to prove that the general fibre over $C_2$ (resp.\ $C_3$)
	is smooth, we prove that the general fibre over the $2$-parameter
	deformation of $V$ with parameters $t_{m^2_{3,1}}$ and
	$t_{m^1_{2,2}}$ (resp.\ $t_{m^1_{3,1}}$ and $t_{m^2_{2,2}}$) is
	smooth. This can be done by applying the jacobian criterion to the
	output of the computer calculations that we will describe below.
\end{proof}

We now illustrate Conjecture~A and Conjecture~B in our example. Let $C_1$, $C_2$,
$C_3$ be the 3 irreducible components of $\Defo V$, whose equations
are given in Proposition~\ref{pro:eq}.
By Remark~\ref{rem:deformations_variety_or_pair} $\Defo (V, \partial V)$ has 3 irreducible components, $\cM_1$, $\cM_2$, $\cM_3$, each of which lie over exactly one of $C_1$, $C_2$, $C_3$.
For each $i \in \{1,2,3\}$ the smooth morphism $\cM_i \to C_i$ induces a surjective linear map $T_0 \cM_i \to T_0 C_i$ of linear representations of the torus $\Spec \CC[M]$.

Let $\alpha$, $\beta$ and $\gamma$ be the three $0$-mutable polynomials with Newton polytope $F$ (see Example~\ref{ex:quadrilateral}).
By comparing the degrees of $T_0 C_1$, $T_0 C_2$, $T_0 C_3$ with the seeds $\widetilde{\seeds}(\alpha)$, $\widetilde{\seeds}(\beta)$, $\widetilde{\seeds}(\gamma)$ in Example~\ref{ex:quadrilateral_conjecture}, we have that $\alpha$ (resp.\ $\beta$, resp.\ $\gamma$) corresponds  to $\cM_1$ (resp.\ $\cM_2$, resp.\ $\cM_3$).

\subsection{Computer computations} \label{sec:computer}
Here we present a proof of Proposition~\ref{pro:eq} which uses the software
Macaulay2 \cite{M2}, in particular the package \verb|VersalDeformations| \cite{ilten_package_article, stevens_computing}.

By observing the degrees of $T^1_V$ (Lemma~\ref{lem:t1}) and the
degrees of $T^2_V$ (Lemma~\ref{lem t2}) it is immediate to see that
each of the 5 equations of $\Defo V \into T^1_V$ can only involve the
following $9$ variables:
\[
t_u \quad
t_{m^1_{3,1}} \quad 
t_{m^1_{2,1}} \quad
t_{m^1_{3,2}} \quad
t_{m^1_{2,2}} \quad
t_{m^2_{3,1}} \quad
t_{m^2_{2,1}} \quad
t_{m^2_{3,2}} \quad
t_{m^2_{2,2}}.
\]
We call the corresponding $9$ degrees of $T^1_V$ the `interesting' degrees of $T^1_V$. This implies that there is a smooth morphism $\Defo V \to G$, where $G$ is a finite dimensional germ with embedding dimension $9$.
We now want to use the computer to determine $G$.

We consider the vector
\[
\begin{pmatrix}
3 \\ 4 \\ 5
\end{pmatrix} \in N = \ZZ^3.
\]
This gives a homomorphism $M \to \ZZ$ and a $\ZZ$-grading on on the
algebra $\CC [\sigma^\vee \cap M] = H^0(V, \cO_V)$, on $T^1_V$, and on
$T^2_V$. We have chosen this particular $\ZZ$-grading because the
corresponding linear projection is injective on the set
$\{u, m^1_{3,1}, m^1_{2,1}, m^1_{3,2}, m^1_{2,2}, m^2_{3,1},
m^2_{2,1}, m^2_{3,2}, m^2_{2,2} \}$,
which will allow us to identify our $9$ variables above with the
corresponding output of Macaulay2 below.  In the following tables we
write down the degrees in $\ZZ$ of the Hilbert basis of
$\sigma^\vee \cap M$, of the interesting degrees of $T^1_V$, and of
the degrees of $T^2_V$.
\[
\begin{array}{cccccccc}
s_1 & z_2 & s_4 & z_3 & s_3 & z_4 & s_2 & u \\ 
9 & 7 & 5 & 3 & 4 & 6 & 8 & 5
\end{array}
\]
\[
\begin{array}{ccccccccc}
-u & -m^1_{3,1} & -m^1_{2,1} & -m^1_{3,2} & -m^1_{2,2} & -m^2_{3,1} & -m^2_{2,1} & -m^2_{3,2} & -m^2_{2,2} \\
-5 & -6 & -1 & 3 & 8 & -7 & -2 & 1 & 6
\end{array}
\]
\[
\begin{array}{ccccc}
4u-s_1 & 4u-s_2 & 5u-s_1-s_2 & 6u-s_1 -s_2 & 9u-2s_1-2s_2 \\
-11 & -12 & -8 & -13 & -11
\end{array}
\]
One can  see that all non-interesting summands of $T^1_V$ have degree $\geq 9$.
Therefore we are interested in the summands of $T^1_V$ with degree between $-7$ and $8$.
Now we run the following Macaulay2 code, which was suggested to us by the referee.

\begin{verbatim}
S = QQ[s1,z2,s4,z3,s3,z4,s2,u,Degrees=>{9,7,5,3,4,6,8,5}];
M = matrix {{s1,z2,u,s2,z4},{z2,s4,z3,z4,s3}};
I = minors(2,M) + ideal (s4*s3-z3^3,z2*s3-z3^2*u,z2*z4-z3*u^2,s1*z4-u^3);
needsPackage "VersalDeformations"
T1 = cotangentCohomology1(-7,8,I)
T2 = cotangentCohomology2(I)
(F,R,G,C) = versalDeformation(gens(I),T1,T2);
G
\end{verbatim}

The output \verb|T1| describes how the equations of $V \into \AA^8$ are perturbed, at the first order, by the coordinates $t_1, \dots, t_9$ of the interesting part of $T^1_V$. From these perturbations one can compute the degrees of these coordinates and discover the following conversion table between our notation and the output of Macaulay2.
\[
\begin{array}{ccccccccc}
t_1 & t_2 & t_3 & t_4 & t_5 & t_6 & t_7 & t_8 & t_9 \\
t_{m^2_{3,1}} & t_{m^2_{3,2}} & t_{m^2_{2,2}} & t_{m^2_{2,1}} & t_{m^1_{3,1}} & t_{m^1_{3,2}} & t_{m^1_{2,2}} & t_{m^1_{2,1}} & t_u
\end{array}\]

The output \verb|G| describes the miniversal deformation space of $V$ with degrees between $-7$ and $8$, i.e.\ the germ $G$ we wanted to study. This implies that $G$ is the germ at the origin of the closed subscheme of $\AA^9$  defined by the following equations:
\begin{gather*}
	t_5 t_9 = 0, \\
	t_1 t_9 = 0, \\
	t_4t_5+t_1t_8 = 0, \\
	t_1 t_5 = 0, \\
	t_2 t_5^2 - t_1^2 t_6 = 0.
\end{gather*}
These equations are those in Proposition~\ref{pro:eq}.
The output \verb|F| gives the equations of the deformation of $V$ over the germ $G$.

\begin{rem}
	The equations of the germ $G$ are only well defined up to a homogeneous change of coordinates whose jacobian is the identity.
	In particular, the quadratic terms of these equations are well defined and can be computed by analysing the cup product $T^1_V \otimes T^1_V \to T^2_V$: this can be done via toric methods \cite{fil, fil2}.
\end{rem}

\bibliography{bib_cofipe}

\begin{bibdiv}
\begin{biblist}

\bib{MR3430830}{article}{
      author={Akhtar, Mohammad},
      author={Coates, Tom},
      author={Corti, Alessio},
      author={Heuberger, Liana},
      author={Kasprzyk, Alexander},
      author={Oneto, Alessandro},
      author={Petracci, Andrea},
      author={Prince, Thomas},
      author={Tveiten, Ketil},
       title={Mirror symmetry and the classification of orbifold del {P}ezzo
  surfaces},
        date={2016},
     journal={Proc. Amer. Math. Soc.},
      volume={144},
      number={2},
       pages={513\ndash 527},
}

\bib{sigma}{article}{
      author={Akhtar, Mohammad},
      author={Coates, Tom},
      author={Galkin, Sergey},
      author={Kasprzyk, Alexander~M.},
       title={Minkowski polynomials and mutations},
        date={2012},
     journal={SIGMA Symmetry Integrability Geom. Methods Appl.},
      volume={8},
       pages={Paper 094, 17},
}

\bib{altmann_minkowski_sums}{article}{
      author={Altmann, Klaus},
       title={Minkowski sums and homogeneous deformations of toric varieties},
        date={1995},
     journal={Tohoku Math. J. (2)},
      volume={47},
      number={2},
       pages={151\ndash 184},
}

\bib{altmann_versal_deformation}{article}{
      author={Altmann, Klaus},
       title={The versal deformation of an isolated toric {G}orenstein
  singularity},
        date={1997},
     journal={Invent. Math.},
      volume={128},
      number={3},
       pages={443\ndash 479},
}

\bib{altmann_one_parameter}{incollection}{
      author={Altmann, Klaus},
       title={One parameter families containing three-dimensional toric
  {G}orenstein singularities},
        date={2000},
   booktitle={Explicit birational geometry of 3-folds},
      series={London Math. Soc. Lecture Note Ser.},
      volume={281},
   publisher={Cambridge Univ. Press, Cambridge},
       pages={21\ndash 50},
}

\bib{altmann_andre_quillen_cohomology}{article}{
      author={Altmann, Klaus},
      author={Sletsj{\o}e, Arne~B.},
       title={Andr\'{e}-{Q}uillen cohomology of monoid algebras},
        date={1998},
     journal={J. Algebra},
      volume={210},
      number={2},
       pages={708\ndash 718},
}

\bib{MR3080816}{article}{
      author={Blanc, J\'{e}r\'{e}my},
       title={Symplectic birational transformations of the plane},
        date={2013},
     journal={Osaka J. Math.},
      volume={50},
      number={2},
       pages={573\ndash 590},
}

\bib{tom_and_jerry}{article}{
      author={Brown, Gavin},
      author={Kerber, Michael},
      author={Reid, Miles},
       title={Fano 3-folds in codimension 4, {T}om and {J}erry. {P}art {I}},
        date={2012},
     journal={Compos. Math.},
      volume={148},
      number={4},
       pages={1171\ndash 1194},
}

\bib{1812.02594}{misc}{
      author={Brown, Gavin},
      author={Reid, Miles},
      author={Stevens, Jan},
       title={{T}utorial on {T}om and {J}erry: the two smoothings of the
  anticanonical cone over $\mathbb{P}(1,2,3)$},
        date={2018},
        note={arXiv:1812.02594},
}

\bib{MR3470714}{article}{
      author={Coates, Tom},
      author={Corti, Alessio},
      author={Galkin, Sergey},
      author={Kasprzyk, Alexander},
       title={Quantum periods for 3-dimensional {F}ano manifolds},
        date={2016},
     journal={Geom. Topol.},
      volume={20},
      number={1},
       pages={103\ndash 256},
}

\bib{toric_lci_are_resolvable}{article}{
      author={Dais, Dimitrios~I.},
      author={Haase, Christian},
      author={Ziegler, G\"{u}nter~M.},
       title={All toric local complete intersection singularities admit
  projective crepant resolutions},
        date={2001},
     journal={Tohoku Math. J. (2)},
      volume={53},
      number={1},
       pages={95\ndash 107},
}

\bib{fil}{misc}{
      author={Filip, Matej},
       title={The {G}erstenhaber product $\operatorname{HH}^2({A})\times
  \operatorname{HH}^2({A}) \to \operatorname{HH}^3({A})$ of affine toric
  varieties},
        date={2018},
        note={arXiv:1803.07486},
}

\bib{fil2}{article}{
      author={Filip, Matej},
       title={Hochschild cohomology and deformation quantization of affine
  toric varieties},
        date={2018},
     journal={J. Algebra},
      volume={506},
       pages={188\ndash 214},
}

\bib{fomin_zelevinsky}{article}{
      author={Fomin, Sergey},
      author={Zelevinsky, Andrei},
       title={Cluster algebras. {I}. {F}oundations},
        date={2002},
     journal={J. Amer. Math. Soc.},
      volume={15},
      number={2},
       pages={497\ndash 529},
}

\bib{fulton_book}{book}{
      author={Fulton, William},
       title={Introduction to toric varieties},
      series={Annals of Mathematics Studies},
   publisher={Princeton University Press, Princeton, NJ},
        date={1993},
      volume={131},
        note={The William H. Roever Lectures in Geometry},
}

\bib{galkin_usnich}{article}{
      author={Galkin, Sergey},
      author={Usnich, Alexandr},
       title={Mutations of potentials},
        date={2010},
     journal={preprint IPMU},
       pages={10\ndash 0100},
}

\bib{M2}{misc}{
      author={Grayson, Daniel~R.},
      author={Stillman, Michael~E.},
       title={Macaulay2},
        note={A software system for research in algebraic geometry. Available
  at \url{http://www.math.uiuc.edu/Macaulay2/}},
}

\bib{MR3350154}{article}{
      author={Gross, Mark},
      author={Hacking, Paul},
      author={Keel, Sean},
       title={Birational geometry of cluster algebras},
        date={2015},
     journal={Algebr. Geom.},
      volume={2},
      number={2},
       pages={137\ndash 175},
}

\bib{MR3415066}{article}{
      author={Gross, Mark},
      author={Hacking, Paul},
      author={Keel, Sean},
       title={Mirror symmetry for log {C}alabi--{Y}au surfaces {I}},
        date={2015},
     journal={Publ. Math. Inst. Hautes \'{E}tudes Sci.},
      volume={122},
       pages={65\ndash 168},
}

\bib{MR3314827}{article}{
      author={Gross, Mark},
      author={Hacking, Paul},
      author={Keel, Sean},
       title={Moduli of surfaces with an anti-canonical cycle},
        date={2015},
     journal={Compos. Math.},
      volume={151},
      number={2},
       pages={265\ndash 291},
}

\bib{gs_intrinsic_slc}{incollection}{
      author={Gross, Mark},
      author={Siebert, Bernd},
       title={Intrinsic mirror symmetry and punctured {G}romov--{W}itten
  invariants},
        date={2018},
   booktitle={Algebraic geometry: {S}alt {L}ake {C}ity 2015},
      series={Proc. Sympos. Pure Math.},
      volume={97},
   publisher={Amer. Math. Soc., Providence, RI},
       pages={199\ndash 230},
}

\bib{1909.07649}{misc}{
      author={Gross, Mark},
      author={Siebert, Bernd},
       title={{I}ntrinsic {M}irror {S}ymmetry},
        date={2019},
        note={arXiv:1909.07649},
}

\bib{MR713093}{incollection}{
      author={Hauser, Herwig},
       title={An algorithm of construction of the semiuniversal deformation of
  an isolated singularity},
        date={1983},
   booktitle={Singularities, {P}art 1 ({A}rcata, {C}alif., 1981)},
      series={Proc. Sympos. Pure Math.},
      volume={40},
   publisher={Amer. Math. Soc., Providence, R.I.},
       pages={567\ndash 573},
}

\bib{MR803194}{article}{
      author={Hauser, Herwig},
       title={La construction de la d\'{e}formation semi-universelle d'un germe
  de vari\'{e}t\'{e} analytique complexe},
        date={1985},
     journal={Ann. Sci. \'{E}cole Norm. Sup. (4)},
      volume={18},
      number={1},
       pages={1\ndash 56},
}

\bib{ilten_package_article}{article}{
      author={Ilten, Nathan~Owen},
       title={Versal deformations and local {H}ilbert schemes},
        date={2012},
     journal={J. Softw. Algebra Geom.},
      volume={4},
       pages={12\ndash 16},
}

\bib{KT}{unpublished}{
      author={Kasprzyk, Alexander},
      author={Tveiten, Ketil},
       title={Maximally mutable {L}aurent polynomials},
        date={2014},
        note={Unpublished notes},
}

\bib{mavlyutov}{misc}{
      author={Mavlyutov, Anvar~R.},
       title={Deformations of toric varieties via {M}inkowski sum
  decompositions of polyhedral complexes},
        date={2009},
        note={arXiv:0902.0967},
}

\bib{nakajima}{article}{
      author={Nakajima, Haruhisa},
       title={Affine torus embeddings which are complete intersections},
        date={1986},
     journal={Tohoku Math. J. (2)},
      volume={38},
      number={1},
       pages={85\ndash 98},
}

\bib{petracci_mavlyutov}{misc}{
      author={Petracci, Andrea},
       title={Homogeneous deformations of toric pairs},
        note={To appear in Manuscripta Math.,
  \url{https://doi.org/10.1007/s00229-020-01219-w}, arXiv:1801.05732},
}

\bib{MR3380790}{incollection}{
      author={Reid, Miles},
       title={Gorenstein in codimension 4: the general structure theory},
        date={2015},
   booktitle={Algebraic geometry in east {A}sia---{T}aipei 2011},
      series={Adv. Stud. Pure Math.},
      volume={65},
   publisher={Math. Soc. Japan, Tokyo},
       pages={201\ndash 227},
}

\bib{stevens_computing}{article}{
      author={Stevens, Jan},
       title={Computing versal deformations},
        date={1995},
     journal={Experiment. Math.},
      volume={4},
      number={2},
       pages={129\ndash 144},
}

\bib{stevens_rolling}{article}{
      author={Stevens, Jan},
       title={Rolling factors deformations and extensions of canonical curves},
        date={2001},
     journal={Doc. Math.},
      volume={6},
       pages={185\ndash 226},
}

\bib{stevens_book}{book}{
      author={Stevens, Jan},
       title={Deformations of singularities},
      series={Lecture Notes in Mathematics},
   publisher={Springer-Verlag, Berlin},
        date={2003},
      volume={1811},
        ISBN={3-540-00560-9},
}

\end{biblist}
\end{bibdiv}

\end{document}